\documentclass[11pt, reqno]{amsart}
\usepackage{amsmath,amsopn,amssymb,amsthm,multicol}
\usepackage[cp1251]{inputenc}
\usepackage[english]{babel}
\usepackage[breaklinks=true,colorlinks=true,linkcolor=blue,citecolor=blue,urlcolor=blue]{hyperref}
\usepackage[dvipdf]{graphicx}

\renewenvironment{proof}[1][Proof]{\textbf{#1.} }
{\ \rule{0.5em}{0.5em}}

\newtheorem{theorem}{Theorem}

\newtheorem{lemma}{Lemma}
\newtheorem{remark}{Remark}
\newtheorem{corollary}{Corollary}

\begin{document}

\title[The dynamics  of a three-dimensional  differential system  \dots]
{On the dynamics of a three-dimensional  differential system
related to the normalized Ricci flow on generalized Wallach spaces}

\author{N.\,A.~Abiev}
\address{N.\,A.~Abiev \newline
Institute of Mathematics NAS KR, Bishkek, prospect Chui, 265a, 720071, Kyrgyz Republic}
\email{abievn@mail.ru}

\begin{abstract} 
We study the behavior of a three-dimensional
dynamical system with respect to some set $S$ given in 3-dimensional euclidian space.
Geometrically such a system  arises from the normalized Ricci flow on some class of generalized Wallach spaces that can be described by a real parameter $a\in (0,1/2)$,
as for~$S$ it  represents the set of invariant Riemannian  metrics of positive sectional curvature on the Wallach spaces. Establishing that  $S$ is bounded by three conic surfaces
and regarding the normalized Ricci flow as an abstract dynamical system
we find out the character of interrelations between
that system and~$S$ for all $a\in (0,1/2)$. These results can cover
some well-known  results, in particular,
they can imply that the normalized Ricci flow evolves all generic invariant Riemannian  metrics with positive sectional curvature into metrics with mixed sectional curvature
on the Wallach spaces  corresponding to the cases $a\in \{1/9, 1/8, 1/6\}$ of generalized Wallach spaces.

\vspace{2mm} \noindent Key words and phrases: Wallach space, generalized Wallach space,
Riemannian metric, normalized Ricci flow, sectional curvature, dynamical system, singular point, phase portrait, invariant set.

\vspace{2mm}

\end{abstract}

\maketitle

\section{Introduction}\label{vvedenie}

The paper is devoted to the study of the dynamics of a system of three nonlinear autonomous ordinary differential equations (given below in~\eqref{three_equat})
relative to some special set  $S\subset \mathbb{R}^3$ (given in~\eqref{DPG}).
A similar system  arises when the evolution of positively curved invariant Riemannian metrics are studying under the influence of the (normalized) Ricci flow on some  homogeneous spaces.
In cases when we are able to describe the set where parameters of such metrics are located,
the problem under consideration can be distracted from its geometric essence
and reformulated as a problem of studying the behavior of an abstract dynamical system's trajectories
with respect to a given abstract set.

The study
 of the normalized Ricci flow equation
\begin{equation}\label{ricciflow}
 \frac{\partial}{\partial t} \bold{g}(t) = -2 \operatorname{Ric}_{\bold{g}}+ 2{\bold{g}(t)}\frac{S_{\bold{g}}}{n}
\end{equation}
 in a Riemannian manifold $\mathcal{M}^n$ was
initiated by R.~Hamilton in~\cite{Ham} and has been attracting an attention during a long time,
where $\bold{g}(t)$ means a $1$-parameter family of Riemannian metrics in~$\mathcal{M}^n$,
$\operatorname{Ric}_{\bold{g}}$ is the Ricci tensor and $S_{\bold{g}}$ is the scalar curvature of  ${\bold{g}}$.

In \cite{AANS1} we started to study~\eqref{ricciflow}
on a class of compact homogeneous spaces  that were introduced and called
{\it three-locally-symmetric}\/ in~\cite{Lomshakov2, Nikonorov2} and   {\it generalized Wallach spaces}\/
in~\cite{Nikonorov1}. A characteristic feature of these spaces is that every generalized Wallach space can be described by a triple of real parameters $a_i\in (0,1/2]$
(for definitions in detail and important properties see~\cite{Nikonorov2, Nikonorov4, Nikonorov1} and references therein).
It should be noted that  classifications of generalized Wallach spaces was recently obtained by Yu.\,G.~Nikonorov in~\cite{Nikonorov4} and independently by Z.~Chen, Y.~Kang and K.~Liang in~\cite{CKL}.

A starting point for our investigations was the possibility of reducing~\eqref{ricciflow} to a system of three autonomous ordinary differential equations (ODEs) with respect to $x_1,x_2,x_3$ due to the fact that generalized Wallach spaces are characterized as compact homogeneous spaces $G/H$ whose isotropy representation
decomposes into a direct sum
$\mathfrak{p}=\mathfrak{p}_1\oplus \mathfrak{p}_2\oplus \mathfrak{p}_3$ of three
$\operatorname{Ad}(H)$-invariant irreducible modules satisfying
$[\mathfrak{p}_i,\mathfrak{p}_i]\subset \mathfrak{h}$ $(i\in\{1,2,3\})$ \cite{Lomshakov2, Nikonorov2}.
For a fixed bi-invariant inner product
$\langle\cdot, \cdot\rangle$ on the Lie algebra $\mathfrak{g}$ of the Lie group $G$,
any $G$-invariant Riemannian metric $\bold{g}$ on $G/H$ is determined by an $\operatorname{Ad} (H)$-invariant inner product
$
(\cdot, \cdot)=\left.x_1\langle\cdot, \cdot\rangle\right|_{\mathfrak{p}_1}+
\left.x_2\langle\cdot, \cdot\rangle\right|_{\mathfrak{p}_2}+
\left.x_3\langle\cdot, \cdot\rangle\right|_{\mathfrak{p}_3}
$,
where $x_1,x_2,x_3$ are positive real numbers
(see~\cite{Lomshakov2, Nikonorov2, Nikonorov4, Nikonorov1} and references therein for more information).
Briefly $\bold{g}$ can be characterized by a triple of positive real numbers $x_1, x_2, x_3$.
In this way
using the expressions
of the Ricci tensor $\operatorname{Ric}_{\bold{g}}$ and
the scalar curvature  $S_{\bold{g}}$ derived in \cite{Nikonorov2} for generalized Wallach spaces the equation~\eqref{ricciflow} was reduced in~\cite{AANS1}
to the  system of ODEs
\begin{equation}\label{three_equatG}
\dot{x}_i(t)= f_i(x_1,x_2,x_3),  \quad i=1,2,3,
\end{equation}
where $x_i=x_i(t)>0$\,  are parameters of the invariant Riemannian metric on a given  generalized  Wallach space and
\begin{eqnarray*}
f_1(x_1,x_2,x_3)&=&-1-a_1x_1 \left( \dfrac {x_1}{x_2x_3}-  \dfrac {x_2}{x_1x_3}- \dfrac {x_3}{x_1x_2} \right)+x_1B,\\
f_2(x_1,x_2,x_3)&=&-1-a_2x_2 \left( \dfrac {x_2}{x_1x_3}- \dfrac {x_3}{x_1x_2} -  \dfrac {x_1}{x_2x_3} \right)+x_2B,\\
f_3(x_1,x_2,x_3)&=&-1-a_3x_3 \left( \dfrac {x_3}{x_1x_2}-  \dfrac {x_1}{x_2x_3}- \dfrac {x_2}{x_1x_3} \right)+x_3B, \\
B&:=&\left( \dfrac {1}{a_1x_1}+\dfrac {1}{a_2x_2}+\dfrac {1}{a_3x_3}- \left( \dfrac {x_1}{x_2x_3}+
\dfrac {x_2}{x_1x_3}+ \dfrac {x_3}{x_1x_2} \right) \right)
\left( \frac{1}{a_1} +\frac{1}{a_2}+ \frac{1}{a_3} \right)^{-1}.
\end{eqnarray*}

Depending on three parameters
$a_1,a_2,a_3$
the qualitative picture of the system~\eqref{three_equatG} is quite varied and complicated.
Taking into account the fact that the volume $\operatorname{Vol}=x_1^{1/a_1}x_2^{1/a_2}x_3^{1/a_3}$ is a first integral of~\eqref{three_equatG} it can be reduced to a planar system
with respect to any two variables, say $x_j(t)$ and~$x_k(t)$ eliminating~$x_i$
on the surface $\operatorname{Vol}\equiv 1$ (the unit volume condition).
For example, eliminating~$x_3$ we obtain
the system of two differential equations
of the type
$\dot{x_1}(t)=\widetilde f_1(x_1,x_2)$, $\dot{x_2}(t)=\widetilde f_2(x_1,x_2)$,
where
$\widetilde f_1(x_1,x_2)\equiv  f_1(x_1,x_2,\varphi(x_1,x_2))$,
$\widetilde f_2(x_1,x_2)\equiv  f_2(x_1,x_2,\varphi(x_1,x_2))$,
$\varphi(x_1,x_2)=x_1^{ -\tfrac {a_3}{a_1}} x_2^{- \tfrac {a_3}{a_2}}$
(note that phase portraits in the coordinate planes
$(x_2, x_3)$ and $(x_1,x_3)$  will be the same due to symmetry in~\eqref{three_equatG}).
Basing on this idea the following results was obtained
on classification of singular points of the system $\dot{x_i}(t)=\widetilde f_i(x_1,x_2)$
in~\cite{Ab_kar3, AANS1, AANS2, AANS3}: there exist at most four distinct singular (equilibria) points for every $(a_1,a_2,a_3)\in (0,1/2]^3$;
non-degenerate singular points are all saddles or nodes
(three saddles and one node or two saddles);
degenerate singular points are all semi-hyperbolic (the nilpotent case never can occur)
excepting an original case $a_1=a_2=a_3=1/4$ with the single linear zero type saddle $(1,1)$
(see~\cite{Dumort, JiangLlibre} for definitions and preliminaries).
In~\cite{Ab2} we established  the topological structure (in the standard topology of $\mathbb{R}^3$) of a special algebraic surface  defined by a symmetric polynomial in variables  $a_1,a_2,a_3$ of degree~$12$  which provides degenerate singular points (a detailed study of that surface was continued
in~\cite{Bat, Bat2}).
In~\cite{Ab1} two-parametric bifurcations of singular points of
$\dot{x_i}(t)=\widetilde f_i(x_1,x_2)$, $i=1,2$, were studied in the case of
two coincided parameters $a_i=a_j$, $i\ne j$.
Similar studies of bifurcations and global dynamics
for three dimensional dynamical systems with three real parameters
can be found in \cite{Llibre, Miy}.
As an example of a multi-parameter bifurcation analysis we refer
readers to the paper~\cite{Tan}, where  a planar system is studied
which depends on five real parameters.

In the sequel we used the system~\eqref{three_equatG} to study the evolution of
positively curved Riemannian metrics on the Wallach spaces and generalized Wallach spaces.
A lot of papers are devoted to homogeneous spaces admitting positive curvature
(sectional  or Ricci curvature).
Results concerning  compact homogeneous spaces  of positive sectional curvature and their classification can be found in the papers  \cite{AW, BB, Be, Wal, Wil, WiZi, XuWolf}.
A wide information concerning  various aspects of this topic can be found in  \cite{Puttmann,Sh2,Valiev, VerZi, Volp1, Volp2, Volp3}.
An interesting class of homogeneous spaces consists of the Wallach spaces
\begin{equation} \label{SWS}
W_6:=\operatorname{SU}(3)/T_{\max},
\ \ W_{12}:=\operatorname{Sp(3)}/\operatorname{Sp(1)}\times \operatorname{Sp(1)}
\times \operatorname{Sp(1)},
\ \ W_{24}:=F_4/\operatorname{Spin(8)}
\end{equation}
that admit invariant Riemannian metrics of positive sectional curvature (see~\cite{Wal} for details).

In general the Ricci flow need not preserve positivity
of curvature.
M.-W.~Cheung and N.\,R.~Wallach
proved in \cite{ChWal} (see~Theorem~2 therein)
that the (normalized) Ricci flow deforms certain
metrics of positive sectional curvature into metrics with mixed sectional  curvature for each homogeneous Riemannian manifold in~\eqref{SWS}.
C.~B\"ohm and B.~Wilking
showed that even the positivity
of the Ricci curvature may not be preserved
for metrics with positive sectional curvature.
They proved that on $W_{12}$
the (normalized) Ricci flow evolves some  metrics of positive sectional curvature into metrics with mixed Ricci curvature (see Theorem~3.1 in \cite{Bo}).
The same assertion was proved for $W_{24}$ in Theorem 3 of~\cite{ChWal}.
The above mentioned results of \cite{Bo, ChWal}
concerning sectional curvature
as well as some of their results relating to evolution of
metrics with positive Ricci curvature were generalized
in~\cite{AN}.
In particular,  we proved  that  on each Wallach space~\eqref{SWS}
the normalized Ricci flow evolves all generic metrics  with positive sectional curvature into metrics with mixed sectional curvature
(see Theorem~1  in~\cite{AN});  on $W_{12}$ and $W_{24}$ the normalized Ricci flow evolves all generic metrics  with positive Ricci curvature into metrics with mixed Ricci curvature  (see Theorem~2 in~\cite{AN}).

By generic metrics we mean metrics with $x_i \ne x_j\ne x_k\ne x_i$, where $i,j,k\in \{1,2,3\}$.
Metrics with two coincided parameters $x_i=x_j$, $i\ne j$,  are called
exceptional and are not of interest to consider.
A geometric explanation of this peculiarity
can be found in \cite{Nikonorov4}. We explain it from the point of view
of dynamical systems in Lemma~\ref{I_i_curves} establishing that
metrics with  $x_i=x_j$ represent invariant sets
of the normalized Ricci flow and, therefore, can admit the primitive evolution.

The main observation that the authors of~\cite{AN}  based on is that each Wallach space
in~\eqref{SWS} can be obtained
in the partial case $a_1=a_2=a_3:=a\in (0,1/2]$ of generalized Wallach spaces
with~$a$  equal to $1/6$, $1/8$ and $1/9$ respectively.
Thus ~\cite{AN} deal with $a\in \{1/6, 1/8, 1/9\}$.
Note also that an analogous question concerning the evolution of Riemannian metrics with positive Ricci curvature
on generalized  Wallach spaces  was considered in~\cite{Ab7} for the special case  $a=1/4$ mentioned above.

For $a_1=a_2=a_3:=a\in (0,1/2]$ the system~\eqref{three_equatG}
 takes
the following form:
\begin{eqnarray}\label{three_equat}
\dot{x}_1(t)=f_1(x_1,x_2,x_3)&=&x_1x_2^{-1}+x_1x_3^{-1} - 2a\,(2x_1^2-x_2^2-x_3^2)(x_2x_3)^{-1}-2, \notag \\
\dot{x}_2(t)=f_2(x_1,x_2,x_3)&=&x_2x_1^{-1}+x_2x_3^{-1} - 2a\,(2x_2^2-x_1^2-x_3^2)(x_1x_3)^{-1}-2,\\
\dot{x}_3(t)=f_3(x_1,x_2,x_3)&=&x_3x_1^{-1}+x_3x_2^{-1} - 2a\, (2x_3^2-x_1^2-x_2^2)(x_1x_2)^{-1}-2. \notag
\end{eqnarray}

In~\cite{AN} we  used a detailed description of invariant metrics of positive sectional curvature on the Wallach spaces \eqref{SWS} given   by F.\,M.~Valiev in  \cite{Valiev}.
We reformulate his results in our notations denoting that set by $S$. Then
\begin{eqnarray}\label{DPG}
S&:=&\left\{(x_1,x_2,x_3)\in (0,+\infty)^3 \,\, | \,\, \gamma_1>0,\, \gamma_2>0,\, \gamma_3>0 \right\}
 \setminus  \left\{(r,r,r)\in \mathbb{R}^3 \,\, | \,\, r>0 \right\}, \\ \label{gamma}
\gamma_i&:=&(x_j-x_k)^2+2x_i(x_j+x_k)-3x_i^2, \quad \{i,j,k\}= \{1,2,3\}.
\end{eqnarray}

Thus we are going to study the behavior of trajectories of the abstract dynamical system~\eqref{three_equat}
 with respect to the formal set~$S$
 distracting from their geometric origin and essence.

For $\{i,j,k\}=\{1,2,3\}$ introduce curves $I_i$ in $\mathbb{R}^3$:
\begin{equation}\label{Isets}
x_i=cp^{-2}, \quad x_j=x_k=p, \quad p>0,
\end{equation}
depending on a parameter $c>0$.
Let  $I=\bigcup_{i=1}^3I_i$ and
 $I'=\bigcup_{i=1}^3I_i'$, where each
$I_i'$ is a part of~$I_i$ considered for $p\ge 1$.

\smallskip

For the boundary, the closure and the exterior of a given set $X$ in $\mathbb{R}^3$
we  use notations  $\partial(X)$, $\overline{X}$ and
$\operatorname{ext}(X)$
respectively  (in the standard topology of $\mathbb{R}^3$).
The main result is contained in the following theorem.

\begin{theorem}\label{theo1}
For $a\in (0,1/2)\setminus \{1/4\}$
the following assertions about trajectories of the system~\eqref{three_equat} hold:
 \begin{enumerate}
 \item
 If $a\in (0,3/14)$ then all trajectories originated in $\overline{S}\setminus I'$
  reach  $\operatorname{ext}(S)$ once and remain there forever;  no trajectory originated  in  $\operatorname{ext}(S)$  reaches $\overline{S}$;
 \item
 If   $a= 3/14$ then all trajectories originated  in $\overline{S}\setminus I$
  reach  $\operatorname{ext}(S)$ once and remain there forever;  no trajectory originated  in  $\operatorname{ext}(S)$  reaches $\overline{S}$;
 \item
If   $a\in (3/14, 1/4)$ then some trajectories  originated  in
$\overline{\operatorname{ext}(S)}$
could get into $S$ and return back;
all trajectories originated  in $S\setminus I$
  reach  $\operatorname{ext}(S)$ once and remain there forever.
 \item
If   $a\in (1/4, 1/2)$ then every trajectory initiated in  $S$ remains there forever;
all trajectories initiated in  $\overline{\operatorname{ext}(S)}$  get into  $S$;
 \item  In the case   $a=1/4$ every trajectory of the system~\eqref{three_equat}
initiated in  $S$ remains there forever;
some trajectories initiated in  $\overline{\operatorname{ext}(S)}$  can get into  $S$.
\end{enumerate}
\end{theorem}

Note that  Theorem~\ref{theo1}
can cover results of~Theorem 1 in \cite{AN} obtained for the Wallach spaces~\eqref{SWS}
(corresponding to the partial cases $a=1/6$, $a=1/8$ and $a=1/9$ respectively).

\section{Auxiliary results}

To prove Theorem~\ref{theo1}   we  need some auxiliary assertions.

\subsection{A description of the set $S$}
Firstly we should give  a detailed proof of structural properties of~$S$
that was noted in  \cite{AN} but omitted.
 Observe that  each $\gamma_i$ in~\eqref{gamma}
is symmetric under the permutations $i\to j\to k\to i$.

\smallskip

\begin{lemma}\label{Valiev_set}
The set $S$ is connected in the standard topology of $\mathbb{R}^3$ and
bounded by the union $\Gamma_1 \cup \Gamma_2 \cup \Gamma_3$ of  pairwise disjoint conic surfaces
$$
\Gamma_i:=\left\{(x_1,x_2,x_3)\in (0,+\infty)^3\,\,|\,\,  \gamma_i=(x_j-x_k)^2+2x_i(x_j+x_k)-3x_i^2=0 \right\}
$$
in $\mathbb{R}^3$, where $\{i,j, k\}=\{1,2,3\}$.
\end{lemma}

\begin{proof}
{\it The sets $\Gamma_1$, $\Gamma_2$ and $\Gamma_3$ are  pairwise disjoint}.
By symmetry it suffices to show  that
 $\Gamma_i$ and $\Gamma_j$ have no
common point  with positive coordinates if $i\ne j$.
Suppose by contrary that some triple $(x_1,x_2,x_3)$ satisfies the  equations $\gamma_i=0$ and $\gamma_j=0$
defining $\Gamma_i$ and $\Gamma_j$. Then
$\gamma_i+\gamma_j=0$ and $\gamma_i-\gamma_j=0$ imply the system
of equations
$(x_i-x_j-x_k)(x_i-x_j+x_k)=0$ and $(x_i-x_j)(x_i+x_j-x_k)=0$
admitting only solutions $x_i=x_j$, $x_k=0$
and $x_i=x_k$, $x_j=0$, where $\{i,j, k\}=\{1,2,3\}$.
But such points are out of our considerations being incompatible
with the conditions $x_1>0$, $x_2>0$ and $x_3>0$.
Hence $\Gamma_i \cap \Gamma_j=\emptyset$ for all  $(x_1, x_2, x_3)\in (0,+\infty)^3$.

\medskip

{\it The set $\Gamma_1 \cup \Gamma_2 \cup \Gamma_3$ is the boundary of $S$}.
Solving  $\gamma_i=0$  with respect to $x_i$ as a quadratic equation
we get two explicit expressions  $x_i=\Psi$  and $x_i=\Phi$ as its \glqq roots``,
where
\begin{equation}\label{Surface_Phi}
\Psi(x_j,x_k)= \frac{x_j+x_k- 2\sqrt{x_j^2-x_jx_k+x_k^2}}{3}, \quad
\Phi(x_j,x_k)= \frac{x_j+x_k+ 2\sqrt{x_j^2-x_jx_k+x_k^2}}{3}.
\end{equation}

Therefore the cone defined by $\gamma_i=0$ consists of two connected components (parts),
say  $\Gamma_i^{-}$ and~$\Gamma_i$,
defined by the equations $x_i=\Psi(x_j, x_k)$
and $x_i=\Phi(x_j, x_k)$ respectively.
It is clear now that the inequality  $\gamma_i>0$ is equivalent to
$\Psi<x_i<\Phi$
as a quadratic inequality with respect to~$x_i$.
It is  easy to observe that $\Psi<0$ at  $x_j>0$ and $x_k>0$
(the surface~$\Gamma_i^{-}$ lies not upper than the plane $x_i=0$).
Therefore $\gamma_i>0$ is equivalent to $0<x_i<\Phi(x_j,x_k)$.
This means that~$S$
is bounded by the plane  $x_i=0$ and the upper part $\Gamma_i$
of the cone $\gamma_i=0$.

By symmetry it is easy to show that the surfaces $\Gamma_j$ and $\Gamma_k$ are also
parts of $\partial (S)$
(see the left panel of Figure~\ref{Gammas0}).
The connectedness of $S$ is obvious.
Lemma~\ref{Valiev_set} is proved.
\end{proof}

\begin{figure}[h]
\centering
\includegraphics[angle=0, width=0.45\textwidth]{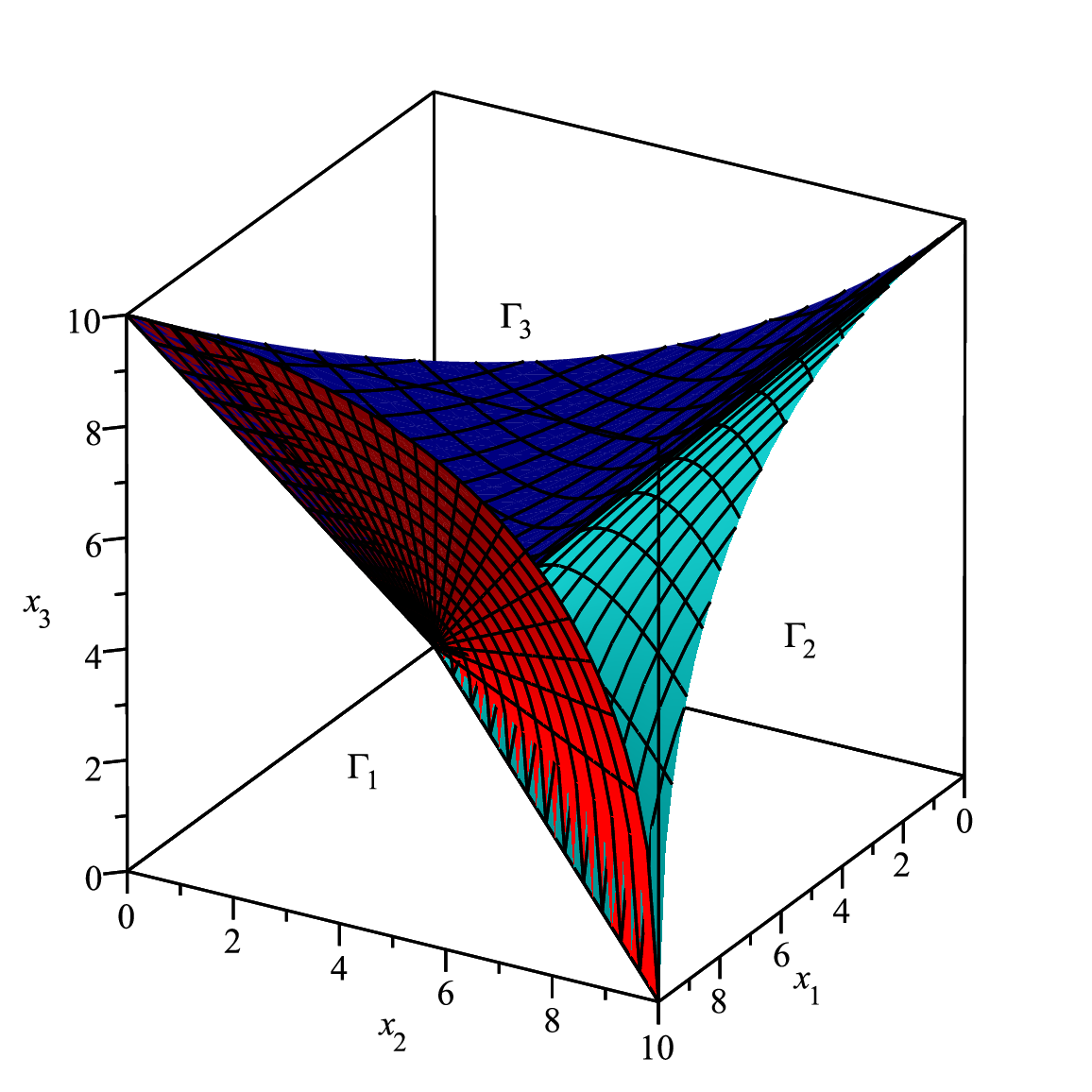}
\includegraphics[angle=0, width=0.45\textwidth]{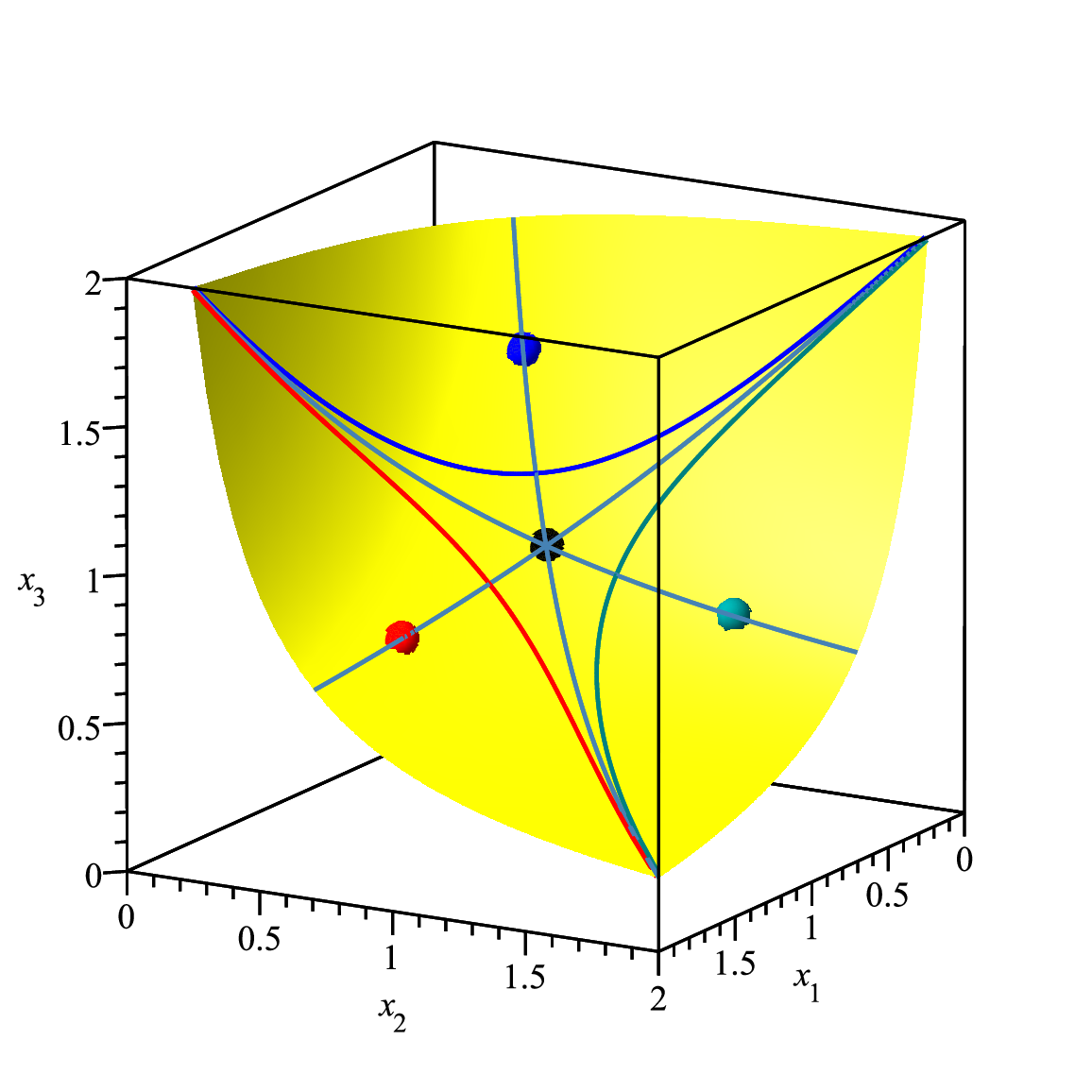}
\caption{The  cones $\Gamma_1, \Gamma_2,  \Gamma_3$ and the surface $\Sigma$ at $c=1$.}
\label{Gammas0}
\end{figure}

\smallskip

Introduce the surface $\Sigma$
determined by the
equation $x_1x_2x_3=c$, where $c$ is any positive constant,

Consider  smooth space curves  $\Sigma\cap \Gamma_i$,  $i=1,2,3$
(these curves and~$\Sigma$ are depicted in the right panel of Figure~\ref{Gammas0}
for the case $c=1$ in red, teal, blue and gold colors respectively).
For each $i=1,2,3$ introduce a planar curve $s_i$ being the orthogonal projection  of  $\Sigma\cap \Gamma_i$ onto the plane $x_3=0$.
Denote by  $D$ the orthogonal projection of
the set $\Sigma \cap S$ onto the plane $x_3=0$.
Lemma~\ref{Valiev_set} implies the following corollary that
can easily be proved by substituting $x_3=cx_1^{-1}x_2^{-1}$ into
the equations $\gamma_i=0$ of the cones $\Gamma_i$ for each $i=1,2,3$.

\begin{corollary}\label{domain_D}
The curves  $s_1,s_2,s_3$ are smooth and defined by the following equations respectively:
\begin{eqnarray}\label{L_i}\notag
l_1&:=&3x_1^4x_2^2-2x_1^3x_2^3-x_1^2x_2^4-2cx_1^2x_2+2cx_1x_2^2-c^2=0,  \\
l_2&:=&3x_1^2x_2^4-2x_1^3x_2^3-x_1^4x_2^2-2cx_1x_2^2+2cx_1^2x_2-c^2=0,\\ \notag
l_3&:=&x_1^4x_2^2-2x_1^3x_2^3+x_1^2x_2^4+2cx_1^2x_2+2cx_1x_2^2-3c^2=0,
\end{eqnarray}
where $x_1>0, x_2>0$. Moreover, $s_1,s_2$ and $s_3$
are pairwise disjoint and bound the domain~$D$.
\end{corollary}

\subsection{Invariant sets and singular points of   \eqref{three_equat}}

Singular points of the system~\eqref{three_equat} are exactly invariant Einstein metrics of the space under studying (see~\cite{Ab7, AANS1, Rod}).
We will reformulate some well-known results
for convenience of readers.
According to  \cite{AANS1, AANS2} for all $a\in (0,1/2)\setminus \{1/4\}$
the system of algebraic equations
$f_1(x_1,x_2,x_3)=0$, $f_2(x_1,x_2,x_3)=0$, $f_3(x_1,x_2,x_3)=0$ given in~\eqref{three_equat}
admits the following four families of one-parametric solutions
being singular points for  \eqref{three_equat} at $\tau>0$:
\begin{equation}\label{lines_equilib}
x_1= x_2= x_3=\tau,   \qquad
x_i=\tau\kappa, ~ x_j=x_k=\tau,   \quad \{i,j,k\}=\{1,2,3\},
\end{equation}
where $\kappa:=\frac{1-2a}{2a}$.
Denote  straight lines in~\eqref{lines_equilib} respectively by $e_0,e_1,e_2,e_3$.
They intersect  the invariant surface~$\Sigma$
at the points
\begin{equation}\label{ququ}
\bold{o_0}:= (\sqrt[3]{c}, \sqrt[3]{c}, \sqrt[3]{c}),\quad \bold{o_1}:=\left(q\kappa, q, q\right), \quad
\bold{o_2}:= \left(q, q\kappa, q\right), \quad  \bold{\bold{o_3}}:=(q, q, q \kappa),
\end{equation}
 where
$q:=\sqrt[3]{c\kappa^{-1}}$
(the case $a=1/6$, $c=1$ is depicted in the right panel of Figure~\ref{Gammas0}
and  corresponding singular points are depicted in black, red, teal and blue colors for visual clarity).

\smallskip

Recall the  curves $I_1$, $I_2$ and $I_3$, introduced in \eqref{Isets}.

\begin{lemma}\label{I_i_curves}
The following assertions are true for the system~\eqref{three_equat}:
\begin{enumerate}
 \item  Every surface  $x_1x_2x_3-c=0$
is invariant, $c>0$;
 \item  At a fixed $c$  every curve $x_i=ct^{-2}$,
 $x_j=x_k=t$
is invariant
and intersects the only cone $\Gamma_i$ at the unique point $x_i=ct_0^{-2}$, $x_j=x_k=t_0$,
approaching the cones~$\Gamma_j$ and $\Gamma_k$ as close  as we like,
where $t_0=\sqrt[3]{6c}/2$, $\{i,j,k\}=\{1,2,3\}$.
\end{enumerate}
\end{lemma}

\begin{proof}
$(1)$
We can establish a more general fact that any function $U(x_1,x_2,x_3):=x_1^{d_1}x_2^{d_2}x_3^{d_3}-C$
can define an invariant surface of the abstract dynamical system~\eqref{three_equatG} for any real numbers $d_1, d_2$ and $d_3$.
Actually we have
$f_i(x_1,x_2,x_3) = -2x_i\left(\bold{r_i}-\frac{S}{n}\right)$
 in~\eqref{three_equatG}, where
  $S= d_1{\bf r_1}+d_2{\bf r_2}+d_3{\bf r_3}$, $n=d_1+d_2+d_3$
 (in the  geometric context $d_i$ means the dimension of the module $\mathfrak{p}_i$ in the decomposition
$\mathfrak{p}=\mathfrak{p}_1\oplus \mathfrak{p}_2\oplus \mathfrak{p}_3$ of
 isotropy representation of a given generalized Wallach space,
  ${\bf r_i}$ and $S$  respectively mean the principal Ricci curvatures and the scalar curvature
of the Riemannian metric),  $i=1,2,3$.
It follows then
$$
\sum_{i=1}^3{\frac{\partial U}{\partial x_i}f_i}=(U+C) \sum_{i=1}^3{\frac{d_i}{x_i}f_i}=
-2(U+C) \sum_{i=1}^3{d_i\left(\bold{r_i}-\frac{S}{n}\right)} = 0.
$$

The parameters $a_i$ are connected with $d_i$ by relations $d_i=A/a_i$ for some positive  $A$.
Obviously  $d_1=d_2=d_3$ in the  partial case $a_1=a_2=a_3$ of the system~\eqref{three_equatG}, therefore,
without loss of generality we can assume that~\eqref{three_equat} admits
an one-parametric family of invariant surfaces~$x_1x_2x_3-c=0$.

\smallskip

$(2)$
Due to symmetry consider  the curve  $I_3$ only.
Substitute $x_1=x_2=t$, $x_3=ct^{-2}$ into $f_1,f_2$ and $f_3$ in~\eqref{three_equat}.
Then as  calculations show
$$
f_1\big(t,t,ct^{-2}\big)=f_2\big(t,t,ct^{-2}\big)=
-\frac{t^3}{2c}\,f_3\big(t,t,ct^{-2}\big)=
\frac{(t^3-c)\big((1-2a)t^3-2ac\big)}{3ct^3}.
$$
Therefore  we have   stationary trajectories
at  $t= \sqrt[3]{c}$  and $t=q$.
The identities
$\dfrac{dx_2}{dx_1}\equiv 1$ and $\dfrac{f_2\big(t,t,ct^{-2}\big)}{f_1\big(t,t,ct^{-2}\big)} \equiv 1$
implies that   $x_1=x_2=t$, $x_3=ct^{-2}$ is  a trajectory of~\eqref{three_equat}
for all $t>0$ such that $t\ne \sqrt[3]{c}$ and $t\ne q$.
By analogy
$ -2ct^{-3}\equiv  \dfrac{d(ct^{-2})}{dt}=\dfrac{dx_3}{dx_1}= \dfrac{f_3\big(t,t,ct^{-2}\big)}{f_1\big(t,t,ct^{-2}\big)} \equiv -2ct^{-3}$.
Therefore the curve  $I_3$ is indeed  invariant under \eqref{three_equat}.
As for  the curves $I_1$ and $I_2$
they are invariant as well  by symmetry ($I_1$, $I_2$ and $I_3$ are shown
in the right panel of Figure~\ref{Gammas0} in aquamarine color).

\smallskip

It is  easy to check that the equation $\gamma_3(t,t,ct^{-2})=0$
equivalent to $4t^3-3c=0$  admits the only   root $\sqrt[3]{6c}/2$ and
$\lim_{t\to +\infty}\gamma_1(t,t,ct^{-2})=\lim_{t\to +\infty}\gamma_2(t,t,ct^{-2})=0$.
Lemma~\ref{I_i_curves} is proved.
\end{proof}

\begin{corollary}\label{c_i_curves}
The projections
$c_1:=\left\{(x_1,x_2)\,|\,  x_1=cx_2^{-2} \right\}$,
$c_2:=\left\{(x_1,x_2)\,|\,  x_2=cx_1^{-2} \right\}$ and
$c_3:=\left\{(x_1,x_2)\,|\,  x_2=x_1 \right\}$
of the space curves $I_i$
onto the plane $x_3=0$
are  invariant sets of the system~\eqref{rnrf_sc}, where $i=1,2,3$
(see~ also \cite{AN}).
$I_1',I_2'$ and $I_3'$ has respectively the projections
$c_1':=\left\{(x_1,x_2)\in c_1\,|\, 0<x_1\le 1  \right\}$,
$c_2':=\left\{(x_1,x_2)\in c_2\,|\,  x_1\ge 1 \right\}$ and
$c_3':=\left\{(x_1,x_2)\in c_3\,|\,   x_1\ge 1\right\}$.
\end{corollary}

\begin{lemma}\label{sing_points_out_S}
 For $a\in (0,1/2)$ the following assertions hold
for the singular points  $\bold{o_1},\bold{o_2},\bold{o_3}$ and $\bold{o_0}$ of the system~\eqref{three_equat}
given in~\eqref{ququ}:~
$\bold{o_0}\in \Sigma\cap S$ and
\begin{itemize}
 \item [(1)]
If $a\in (0,3/14)$  then $\bold{o_1}, \bold{o_2}, \bold{o_3} \in \Sigma\cap \operatorname{ext}(S)$;
\item [(2)]
 If   $a=3/14$ then $\bold{o_1}, \bold{o_2}, \bold{o_3}\in  \Sigma\cap \partial(S)$,
 moreover, $\bold{o_1}\in \Sigma\cap \Gamma_1$, $\bold{o_2}\in\Sigma\cap \Gamma_2$, $\bold{o_3}\in \Sigma\cap \Gamma_3$.
 \item [(3)]
If $a\in (3/14, 1/2)$ then $\bold{o_1}, \bold{o_2}, \bold{o_3} \in \Sigma\cap S$, moreover
  $\bold{o_1}=\bold{o_2}=\bold{o_3}=\bold{o_0} \in \Sigma\cap S$ if and only if $a=1/4$.
 \end{itemize}
\end{lemma}

\begin{proof}
All  inequalities $\gamma_i=(x_j-x_k)^2+2x_i(x_j+x_k)-3x_i^2>0$ defining the set $S$ hold   for $(x_1,x_2,x_3)=(\tau,\tau,\tau)$ trivially, and hence
$e_0\subset S$. In particular, $\bold{o_0}=(1,1,1)\in \Sigma\cap S$.

Due to symmetry under the permutations $i\to j\to k\to i$
it suffices to consider the  straight line
$x_i=\tau \kappa$, $x_j=x_k=\tau>0$ denoted by $e_i$ and
given in~\eqref{lines_equilib} to study location of $\bold{o_1},\bold{o_2}$ and~$\bold{o_3}$.
Note that $\Gamma_i\subset \partial(S)$  by Lemma~\ref{Valiev_set}
and  $\Gamma_i$ is defined by the equation $x_i=\Phi(x_j,x_k)$ as in~\eqref{Surface_Phi}.
Substituting $x_i=\tau \kappa$, $x_j=x_k=\tau$
into
$x_i-\Phi(x_j,x_k)$ yields
\begin{equation}\label{signs_of_gamma}
x_i-\Phi(x_j,x_k)= x_i-\frac{x_j+x_k+ 2\sqrt{x_j^2-x_jx_k+x_k^2}}{3}=
\frac{(1-2a)\,\tau}{2a}-\frac{4\tau}{3}=\frac{(3-14a)\,\tau}{6a}.
\end{equation}

\smallskip

$(1)$ {\it The case  $a\in (0,3/14)$}.
Actually,  we can prove  a more general assertion that
no point of~$e_i$   can  belong to
$\overline{S}=S\cup \partial(S)$ for all $\tau>0$.
Indeed $x_i-\Phi(x_j,x_k)>0$ if  $0<a<3/14$ and  $\tau>0$
meaning that  every point of $e_i$ satisfies $\gamma_3<0$,
in other words,
 belongs to  $\operatorname{ext}(S)$.
Since $e_1,e_2,e_3\subset \operatorname{ext}(S)$ for all $\tau>0$
then, in particular,
$\bold{o_1},\bold{o_2},\bold{o_3}\in \Sigma\cap \operatorname{ext}(S)$
 corresponding to the fixed value $\tau=q:=\sqrt[3]{c\kappa^{-1}}$
(see the right panel of the mentioned Figure~\ref{Gammas0}).

\smallskip

$(2), (3)$ {\it The case $a\in [3/14, 1/2)$}.
It follows from \eqref{signs_of_gamma} immediately that
$e_1, e_2, e_3\subset \Gamma_3$ if $a=3/14$ and $e_1,e_2,e_3\subset S$ else.
Therefore $\bold{o_1}, \bold{o_2}, \bold{o_3} \in \Sigma\cap S$ at $a\in (3/14, 1/2)$
and $\bold{o_1}\in \Sigma\cap \Gamma_1$, $\bold{o_2}\in\Sigma\cap \Gamma_2$, $\bold{o_3}\in \Sigma\cap \Gamma_3$
at $a=3/14$.
Obviously, the lines $e_1, e_2$ and $e_3$ must coincide with $e_0$
at~$a=1/4$ and, hence, belong to $S$
implying $\bold{o_1}=\bold{o_2}=\bold{o_3}=\bold{o_0}\in \Sigma\cap S$.
Lemma~\ref{sing_points_out_S}   is proved.
\end{proof}

\subsection{The vector field  $\bold{V}$ and the normal $\nabla \gamma_i$ of the cone  $\Gamma_i$}
We also need to consider
the standard inner product
$$
(\bold{V}, \nabla \gamma_i)=f_1\frac{\partial \gamma_i}{\partial x_1}+f_2\frac{\partial \gamma_i}{\partial x_2}+f_3\frac{\partial \gamma_i}{\partial x_3}
$$
on the arbitrary cone~$\Gamma_i$,
where  $\bold{V}=(f_1,f_2,f_3)$ is a vector field associated with the system \eqref{three_equat} and
$\nabla \gamma_i:=\left(\tfrac{\partial \gamma_i}{\partial x_1}, \tfrac{\partial \gamma_i}{\partial x_2}, \tfrac{\partial \gamma_i}{\partial x_3}\right)
$ is the normal vector of the surface $\Gamma_i$, $i=1,2,3$.

\begin{lemma}\label{normal_inside}
On every point of the  boundary $\partial(S)$ of the domain $S$
the normal vector is directed inside $S$.
\end{lemma}

\begin{proof}
Due to  symmetry it suffices to take the conic surface
$\Gamma_i\subset \partial(S)$  defined by
$x_i=\Phi(x_j,x_k)$  in Lemma~\ref{Valiev_set}.
Observe that $x_i> \max(x_j,x_k)$ for  every point
$(x_j,x_i,x_k)\in \Gamma_i$.

Indeed, suppose $x_j\ge x_k$ without loss of generality.
Then $x_i> x_j$ is equivalent to the irrational inequality
$2\sqrt{x_j^2-x_jx_k+x_j^2}> 2x_j-x_k$
and can easily be justified  by squaring both sides
because   $2x_j-x_k=x_j+(x_j-x_k)>0$.
Since $\dfrac{\partial \gamma_i}{\partial x_i}= 2(x_j+x_k)-6x_i$
the inequality   $x_i > \max(x_j,x_k)$ implies
$$
\dfrac{\partial \gamma_i}{\partial x_i}=2(x_j-x_i)-2x_i+2(x_k-x_i)<0.
$$

This means that  the normal vector
$\nabla \gamma_i:=\left(\tfrac{\partial \gamma_i}{\partial x_1}, \tfrac{\partial \gamma_i}{\partial x_2}, \tfrac{\partial \gamma_i}{\partial x_3}\right)
$
of the surface $\Gamma_i$  is directed inside the domain~$S$ on the part of its border $\Gamma_i\subset \partial(S)$.
The same property can be proved for other components $\Gamma_j$ and $\Gamma_k$ of the
set $\partial(S)$.
Lemma~\ref{normal_inside} is proved.
\end{proof}

\medskip

By direct calculations we can obtain the following
expression for $(\bold{V}, \nabla \gamma_i)$:
\begin{multline}\label{eqH}
(\bold{V}, \nabla \gamma_i)=\frac{2}{3}\,
\big[-4a\big(-3x_i^4+x_j^4+x_k^4\big)-4a\big(x_i^2(x_j^2+x_k^2)-x_j^2x_k^2\big)\\
+4a(-x_i+x_j+x_k)
+2a\big(x_j^3x_k+x_k^3x_j-x_i(x_j^3+x_k^3)-x_i^3(x_j+x_k)\big)+ x_i(x_j^3+x_k^3) \\+x_j^3x_k+x_k^3x_j
+2\big(x_i^2(x_j^2+x_k^2)-x_j^2x_k^2\big)
-2\big(-2x_i+x_j+x_k\big)-3x_i^3(x_j+x_k)\big].
\end{multline}

By homogeneity of the equation $\gamma_i=0$
 every point
of the cone~$\Gamma_i$ can be represented as a point
\begin{equation}\label{numu}
x_i=\nu t, \quad  x_j=\mu t, \quad  x_k=t, \quad t>0
\end{equation}
of some straight line,
where    $\nu>1$ and $\mu=1-\nu+2\sqrt{\nu(\nu-1)}$ according to Lemma~\ref{Valiev_set}.
Observe that $\mu>0$.
Substituting    \eqref{numu} into  \eqref{eqH}
we obtain
\begin{equation}\label{expres_W}
(\bold{V}, \nabla \gamma_i)=12\,t\,\nu(\nu-1)\big(G(\nu)a-H(\nu)\big),
\end{equation}
where~ $G(\nu)=(12\nu-15)\sqrt{\nu(\nu-1)}-12\nu^2+21\nu-5$ and
$H(\nu)=2^{-1}\left(3\nu-1-3\sqrt{\nu(\nu-1)}\right)$.

\begin{lemma}\label{function_F}
At   $\nu\ge 1$ the following assertions hold for the function $F(\nu):=H(\nu)/G(\nu)$:
\begin{itemize}
\item[(1)]
The function $F(\nu)$ is monotonic on each of
intervals $(1,4/3)$ (decreases) and  $(4/3, +\infty)$ (increases), satisfies
$\lim\limits_{\nu \to +\infty} F(\nu)=1/4$ and
achieves its biggest and smallest  values $1/4$  and $3/14$
at $\nu=1$ and $\nu=4/3$ respectively
(the graph of $F(\nu)$ is depicted in  the left panel of~Figure~\ref{graph_F}).

\item [(2)]
For $a\in (3/14,1/4)$ the equation $F(\nu)=a$ admits exactly two distinct real roots $\nu_1=\nu_1(a)>1$ and $\nu_2=\nu_2(a)>1$ given by the formula
(see the right panel of Figure~\ref{graph_F} for their graphs):
\begin{equation}\label{roots_nu}
\nu_{1,2}(a)=\frac{3\,(2a-1)(10a-1)\pm \sqrt{3\,(14a-3)(10a-1)^3}}{48a\,(4a-1)}.
\end{equation}
In addition,\,  $\nu_1(a)= \nu_2(a)=4/3$ at $a=3/14$ (illustrated by a point $K$),
$\nu_1(a)< \nu_2(a)$ for $a\in (3/14,1/4)$
 and\,
$\lim_{a\to 1/4-0}\nu_1(a)=1$, $\lim_{a\to 1/4-0}\nu_2(a)=+\infty$.
\end{itemize}
\end{lemma}

\begin{proof}
It is clear that $H(\nu)>0$ and $G(\nu)>0$ for all $\nu\ge 1$.
Indeed $H(\nu)=0$ implies the equation
$9\nu(\nu-1)=(3\nu-1)^2$ admitting a negative root $\nu=-1/3$ only.
Therefore $H(\nu)$ has no roots and hence preserves its sign on $\nu\ge 1$.
Since $H(1)=1>0$ then $H(\nu)>0$ for all $\nu\ge 1$.
Analogously squaring $G(\nu)=0$ we obtain the equation
$24\nu^2-15\nu-25=0$ no root of that can satisfy it.
Consequently,  $F(\nu)>0$ at  $\nu\ge 1$.

\smallskip

$(1)$ {\it The smallest and the biggest values of $F(\nu)$}.
The function $F(\nu)$ attains its  smallest value~$3/14$  at $\nu=4/3$.
Indeed the local extrema's necessary condition
$F'(\nu)=0$ is equivalent to the irrational equation
$(24\nu^2-20\nu+2)\sqrt{\nu(\nu-1)} = 24\nu^3-32\nu^2-9\nu$
which implies  $\nu(3\nu-4)=0$ after squaring.
Observe that the value  $\nu_0=4/3$ is the only zero of
the derivative $F'(\nu)$, whereas the second derivative $F''(\nu_0)=27/392>0$.
Therefore $\nu_0$ is a local minimum point for $F$.
Moreover, it provides the smallest value $2/14$ of $F(\nu)$ for all $\nu\ge 1$ since  $F(1)=1/4>3/14=F(\nu_0)$
and $\lim\limits_{\nu \to +\infty} F(\nu)=1/4$ due to
\begin{eqnarray*}
\lim_{\nu \to +\infty} H(\nu)&=&\{\,\mbox{replacing}~ \nu=y^{-1}\}=
\lim_{y \to 0}\frac{3-y-3\sqrt{1-y}}{2y}= \{\,\mbox{L'H\^{o}pital's rule}\,\}=1/4,\\
\lim_{\nu \to +\infty} G(\nu)&=&\lim_{y \to 0}\frac{(12-15y)\sqrt{1-y}-12+21y-5y^2}{y^2}= \{\,\mbox{L'H\^{o}pital's rule}\,\}=1.
\end{eqnarray*}

Clearly $1/4$ is the biggest value of $F(\nu)$  for $\nu\ge 1$.

\smallskip

$(2)$ {\it  Roots of the equation  $F(\nu)=a$}.
According to the proved  part of  Lemma~\ref{function_F}
 the equation $F(\nu)=a$ admits exactly two
distinct real roots $\nu_1=\nu_1(a)$ and $\nu_2=\nu_2(a)$, say $\nu_1<\nu_2$,  for every $a\in (3/14, 1/4)$.
We can find explicit expressions for them.
Since $G,H>0$ for $\nu>1$  the  irrational equation $F=a$
is equivalent to a new one $0=G^2a^2-H^2:=M-N\sqrt{\nu(\nu-1)}$,
where
$M:=(288\nu^4-1008\nu^3+1146\nu^2-435\nu+25)a^2-9\nu^2/2+15\nu/4-1/4$ and
$N:=(288\nu^3-864\nu^2+750\nu-150)a^2-9\nu/2+3/2$.
Raising $M=N\sqrt{\nu(\nu-1)}$
to second power we obtain in consequence
two quadratic equations
\begin{eqnarray*}
24a(4a+1)\nu^2-3(10a+1)(2a+1)\nu -(10a+1)^2&=&0, \\
24a(4a-1)\nu^2-3(10a-1)(2a-1)\nu -(10a-1)^2&=&0.
\end{eqnarray*}
Observe that  the first of them  generates extraneous roots for $Ga-H=0$.
Therefore $\nu_1(a)$ and $\nu_2(a)$ can satisfy the second one only and hence can be found by \eqref{roots_nu}.
Note that $\nu_1(a)$ is also extraneous if $a>1/4$ satisfying  $M+N\,\sqrt{\nu(\nu-1)}=0$.
 Lemma~\ref{function_F} is proved.
\end{proof}

\begin{figure}[h!]
\centering
\includegraphics[angle=0, width=0.45\textwidth]{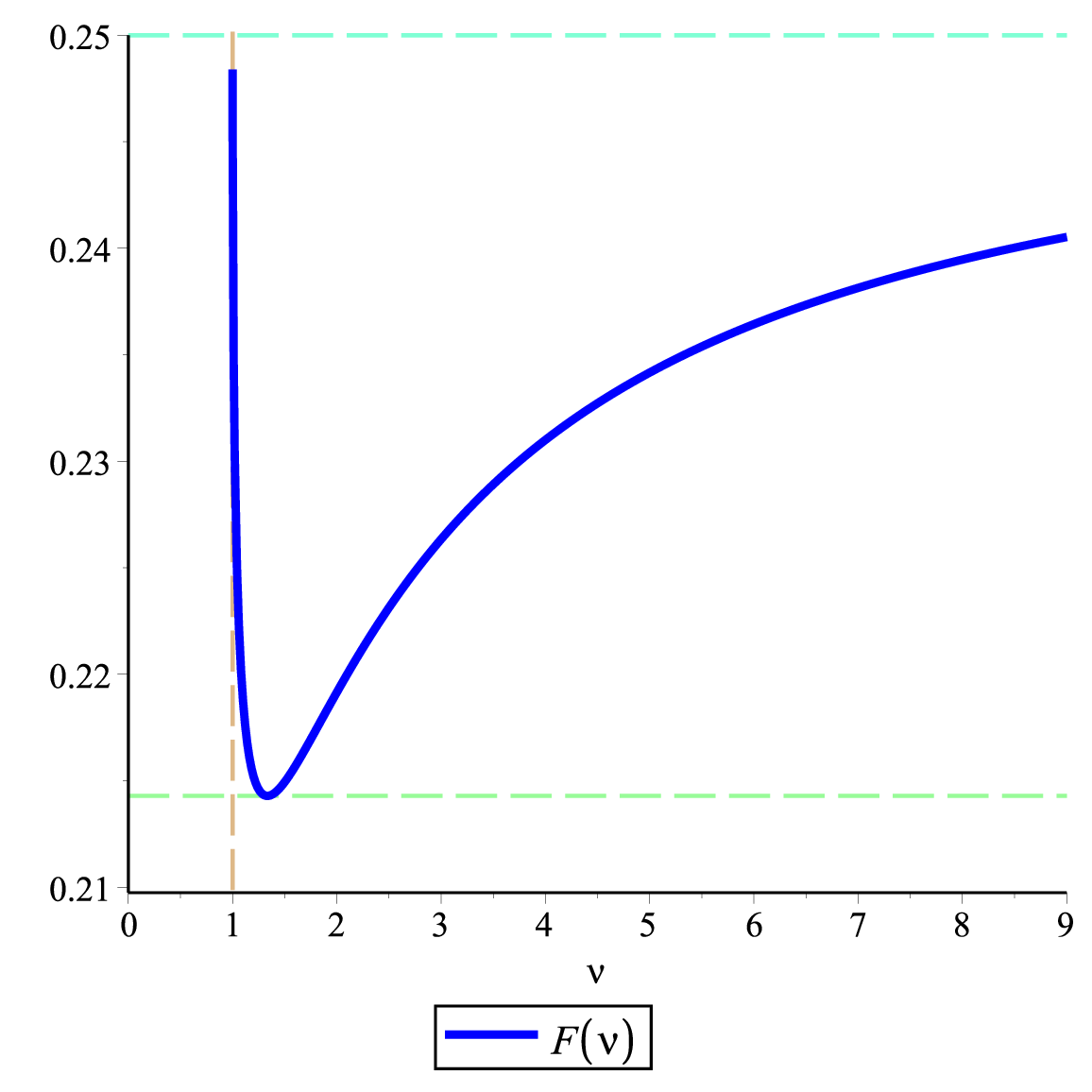}
\includegraphics[angle=0, width=0.45\textwidth]{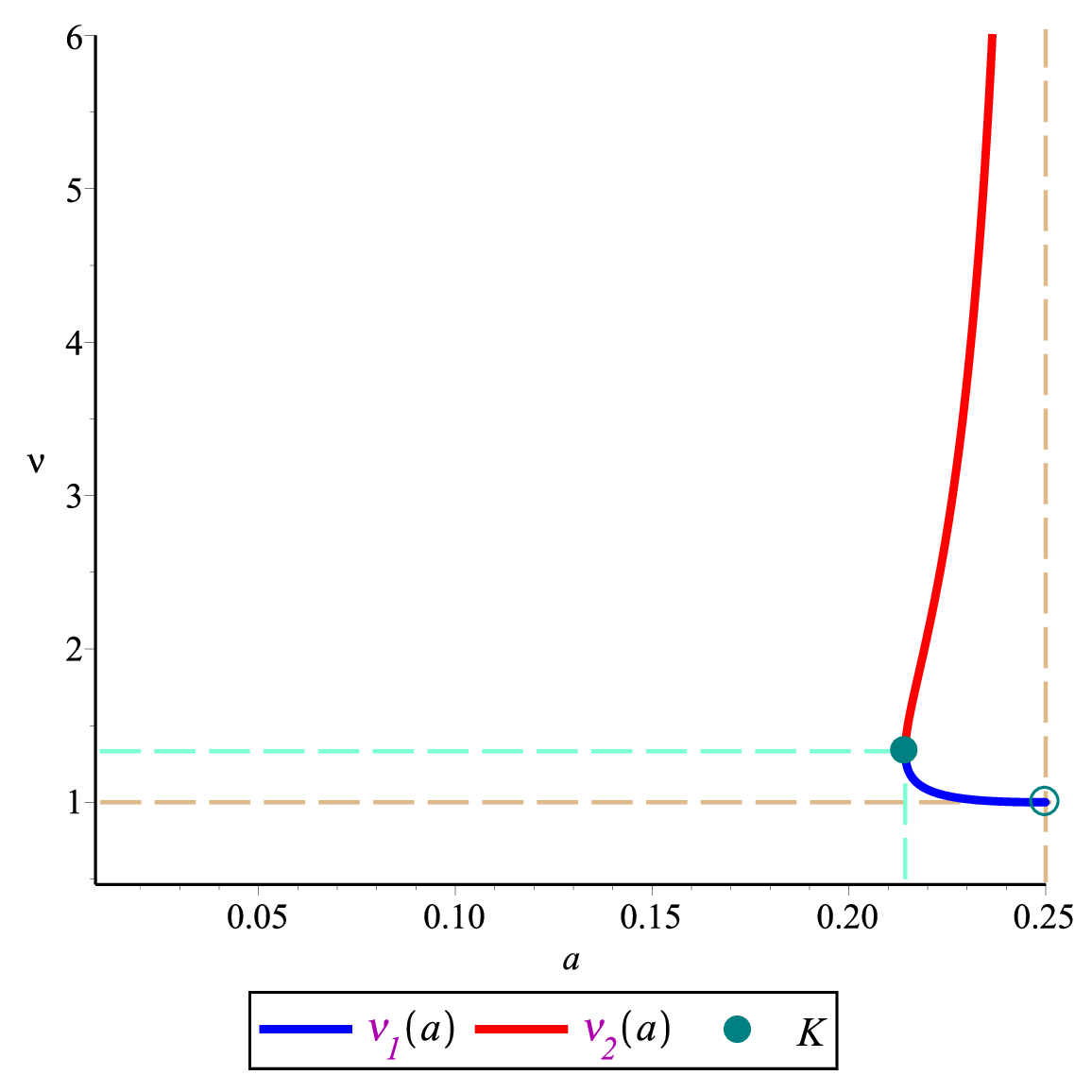}
\caption{The graphs of functions  $F(\nu)$,  $\nu_1(a)$ and $\nu_2(a)$.}
\label{graph_F}
\end{figure}

\subsection{Stable, unstable and center manifolds of singular points of~\eqref{three_equat}}

Denote by $\bold{J}_i$ the Jacobian matrix of~\eqref{three_equat}
 corresponding to its singular point~$\bold{o_i}$, $i=0,1,2,3$, and
use the standard notations $W_i^s$, $W_i^u$ and $W_i^c$ respectively  for the stable, unstable and
center manifolds of~$\bold{o_i}$.
Denote by $E_i^s$, $E_i^u$ and $E_i^c$ the corresponding eigenspaces
tangent to $W_i^s$, $W_i^u$ and $W_i^c$ at $\bold{o_i}$.

\begin{lemma}\label{lefolds}
The following assertions hold for the singular points \eqref{ququ}  of the system~\eqref{three_equat}:
\begin{enumerate}
\item The case $a<1/4$:
\begin{itemize}
\item[(a)]   $E_0^u=\operatorname{Span}\left\{(-1,0,1),(-1,1,0)\right\}$, $W_0^u=\Sigma$  and
$I_i'\subset W_0^u$ for each $i=1,2,3$;

\item[(b)]
$E_1^s=\operatorname{Span}\left\{(2a-1,a,a)\right\}$,
$E_2^s=\operatorname{Span}\left\{(a,2a-1, a)\right\}$,
$E_3^s=\operatorname{Span}\left\{(a,a,2a-1)\right\}$ and $W_i^s=I_i\setminus I_i'$,   $i=1,2,3$;

\item[(c)]

$E_1^u=\operatorname{Span}\left\{(0,-1,1)\right\}$,
$E_2^u=\operatorname{Span}\left\{(-1,0, 1)\right\}$,
 $E_3^u=\operatorname{Span}\left\{(-1, 1,0)\right\}$;

\end{itemize}

\item The case $a>1/4$:

\begin{itemize}
\item[(d)]   $E_0^s=\operatorname{Span}\left\{(-1,0,1),(-1,1,0)\right\}$,
$W_0^s=\Sigma$ and
$I_i\setminus I_i'\subset W_0^s$ for each $i=1,2,3$;

\item[(f)]
$E_1^u=\operatorname{Span}\left\{(2a-1,a,a)\right\}$,
$E_2^u=\operatorname{Span}\left\{(a,2a-1, a)\right\}$,
 $E_3^u=\operatorname{Span}\left\{(a,a,2a-1)\right\}$ and $W_i^u=I_i'$,   $i=1,2,3$;

\item[(g)]

$E_1^s=\operatorname{Span}\left\{(0,-1,1)\right\}$,
$E_2^s=\operatorname{Span}\left\{(-1,0, 1)\right\}$,
 $E_3^s=\operatorname{Span}\left\{(-1, 1,0)\right\}$;

\end{itemize}
\item
$E_0^c=\operatorname{Span}\left\{(1,1,1)\right\}$,
$E_1^c=\operatorname{Span}\left\{(\kappa,1,1)\right\}$,
$E_2^c=\operatorname{Span}\left\{(1,\kappa, 1)\right\}$,
 $E_3^c=\operatorname{Span}\left\{(1,1,\kappa)\right\}$ for all $a\in (0,1/2)\setminus \{1/4\}$.
\end{enumerate}
\end{lemma}

\begin{proof}
Eigenvectors $(-1,0,1)$, $(-1,1,0)$
correspond to the eigenvalue  $-\frac{4a-1}{q}$
 of the  Jacobian matrix $\bold{J}_0$ of multiplicity $2$.
 For every $c>0$ this vectors define the plane
$x_1+x_2+x_3-3\sqrt[3]{c}=0$ tangent to $\Sigma$
at $\bold{o_0}$. The property of this plane  to be stable or not, of course,
depends on the sign of $4a-1$.
The third eigenvalue of $\bold{J}_0$ equals $0$  and the corresponding eigenvector is
$(1,1,1)$ spanning the line  $x_1=x_2=x_3=t$.

For $i=1,2,3$ the matrix $\bold{J}_i$
has eigenvalues  $\frac{4a-1}{q}$,
$\frac{(4a-1)(2a+1)}{(2a-1)q}$ and $0$.
The eigenvectors corresponding to nonzero eigenvalues
span straight lines
$x_i=q\kappa +(2a-1)t$, $x_j=x_k=at$ and $x_i=0$, $x_j=-x_k=t$
respectively tangent to the stable (unstable) and the unstable (stable) manifolds
of~$\bold{o_i}$ if $a<1/4$ (if $a>1/4$).
The slow manifold of $\bold{o_i}$  has the tangent $x_i=\kappa t$, $x_j=x_k=t$
(all slow eigenspaces $E_i^c$, $i=0,\dots, 4$,  are demonstrated in green color
in the left panel of Figure~\ref{Mani_sing} for $a=1/6$, $c=1$).

In addition,  $I_i\setminus I_i'=W_i^s$ and $I_i'\subset W_0^u$ if $a<1/4$
and $I_i\setminus I_i'\subset W_0^s$ and $I_i'=W_i^u$ if $a>1/4$
(the case $i=3$, $a=1/6$, $c=1$ is depicted in the right panel of Figure~\ref{Mani_sing}).
\end{proof}

\begin{figure}[h!]
\centering
\includegraphics[angle=0, width=0.45\textwidth]
{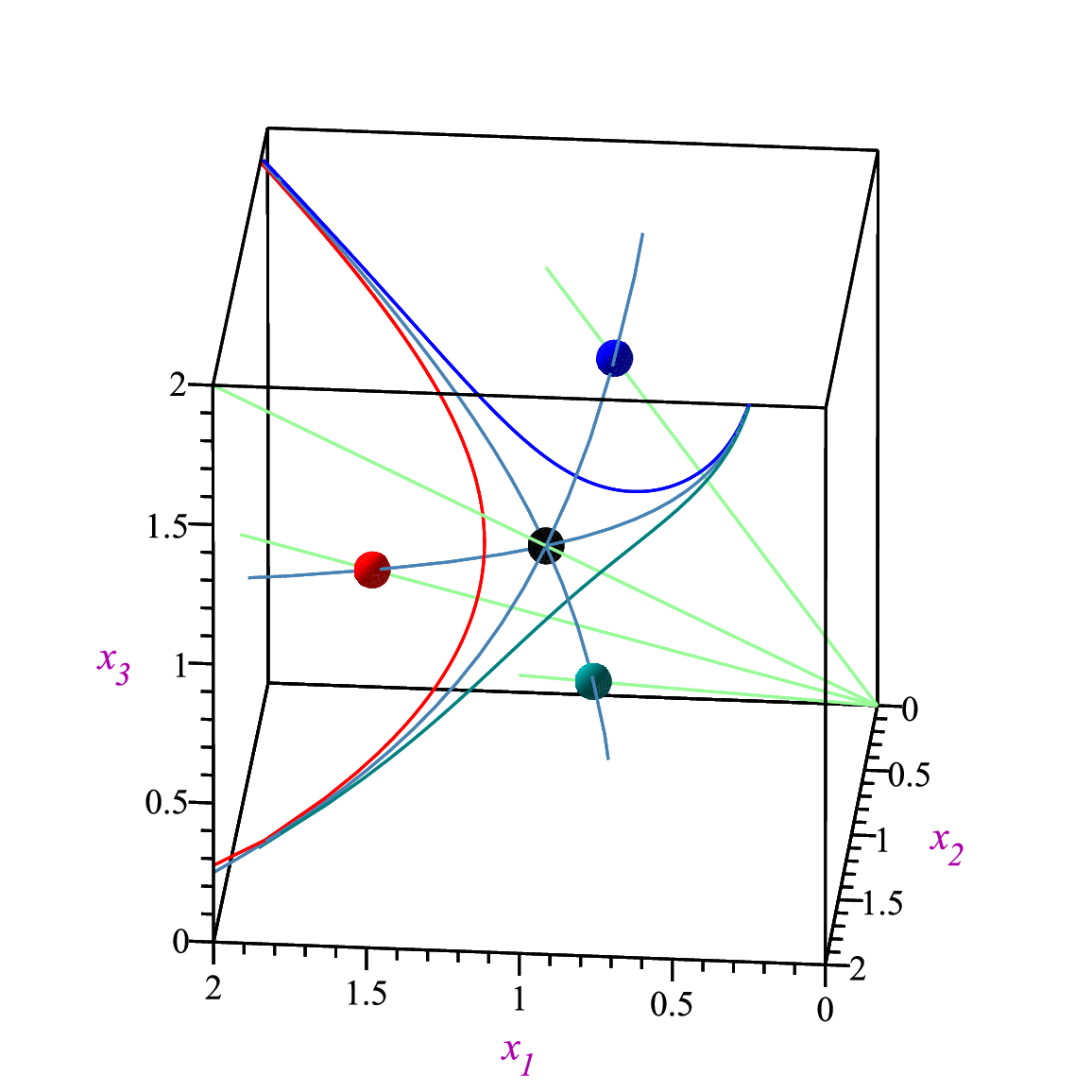}
\includegraphics[angle=0, width=0.45\textwidth]
{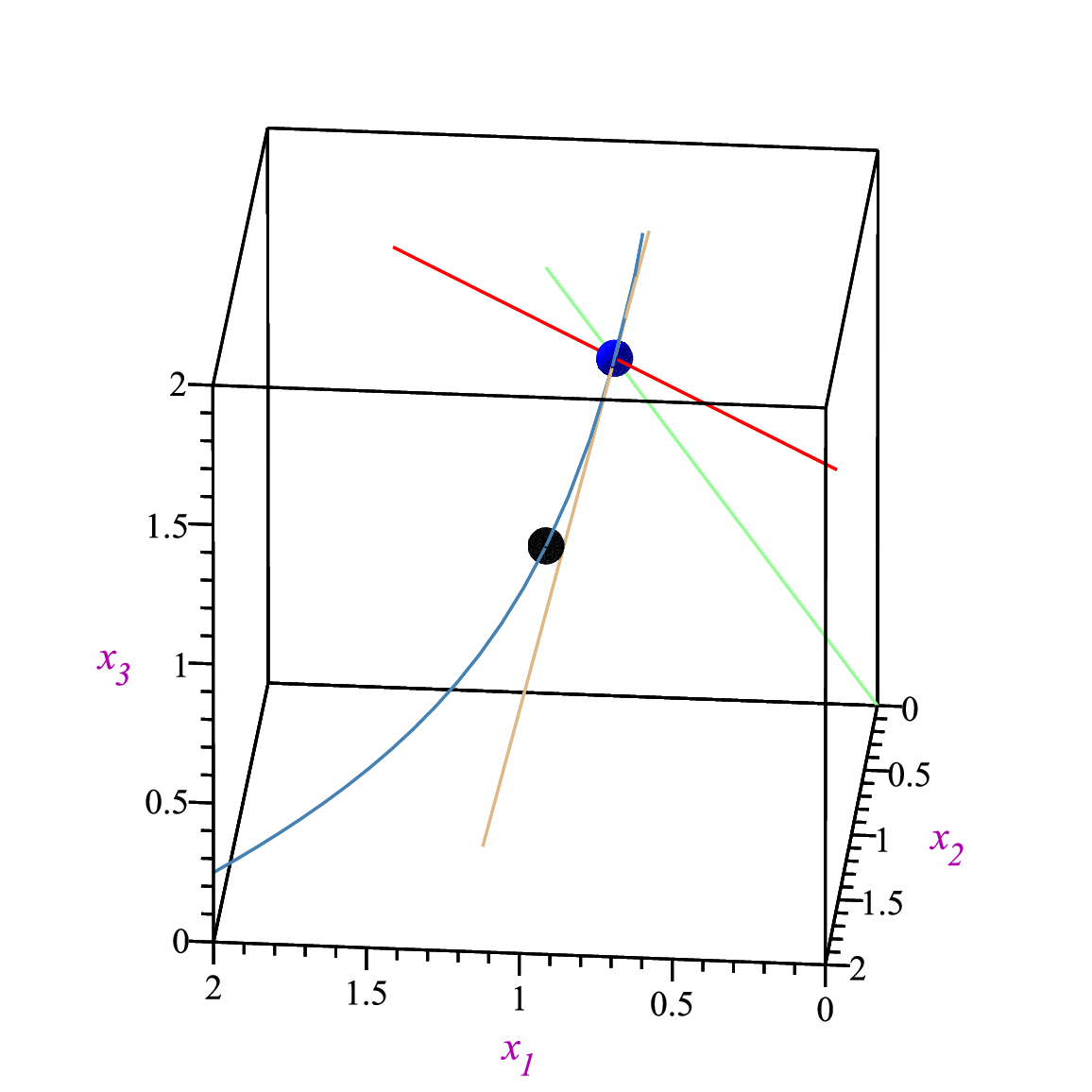}
\caption{The case $a=1/6$, $c=1$: the points $\bold{o_i}$ and $E_i^c$ (the green lines) for $i=0,1,2,3$ (the left panel); $W_3^s=I_3\setminus I_3'$ (the connected component in the steel blue curve originated at $\bold{o_0}$ and containing $\bold{o_3}$), $E_3^s$ (the burlywood straight line)
and $E_3^u$ (the red straight line) (the right panel).}
\label{Mani_sing}
\end{figure}

\begin{remark}\label{refolds}
Since all center manifolds are slow we conclude  that
the dominant motions of trajectories of~\eqref{three_equat}  towards or away
from the singular points $\bold{o_i}$ happen on their strong  stable and unstable manifolds
contained in the invariant set $\Sigma=\Sigma(c)$.
Thus  surfaces $x_1x_2x_3=c$ are most responsible for dynamics of~\eqref{three_equat}
at any fixed $c>0$.
However knowing all the corresponding  subspaces (straight lines) $E_i^s$, $E_i^u$
we do not have a complete information about
manifolds (curves) $W_i^s$, $W_i^u$ except for the invariant curves $I_i$.
According to the theory closeness of $E_i$ to $W_i$ possesses
the local character only nearby the points~$\bold{o_i}$.
Take, for example the case $a=1/8$ considered in~\cite{AN}.
Then trajectories in $S$ can be repelled out by~$\bold{o_0}$ and
attracted by some $\bold{o_i}$ due to influence of $W_i^s=I_i\setminus I_i'$
which intersects $\partial(S)$ transversally by Lemma~\ref{I_i_curves}.
However we do not know will $W_i^u$ \glqq help`` all trajectories leave $S$.
This would be so if the curve $W_i^u$ was outside of $S$ at least at infinity.
Otherwise some trajectories will stay in $S$.
We encountered this problem in~\cite{AN}  and proposed a way to circumvent such difficulties.
Proposition~1 in~\cite{AN} based on the asymptotic idea
allows answer the question will trajectories  ever be outside (or inside) of~$S$ and then stay
there forever or not
without knowing  equations of the curves~$W_i^u$.
Briefly  Proposition~1 asserts the following:
every trajectory of the system~\eqref{three_equat} originated in
$S\setminus I$  reaches the boundary $\partial(S)$   in finite time
and gets into $\operatorname{ext(S)}$ if $\frac{1-2a}{4a}>\frac{1}{2}$ and
respectively every trajectory of~\eqref{three_equat} originated in $\operatorname{ext(S)}\setminus I$
reaches the boundary~$\partial(S)$   in finite time
and gets into $S$ if  $\frac{1-2a}{4a}<\frac{1}{2}$.
That time can be large enough depending
on a chosen initial point.
\end{remark}

\section{Proof of Theorem~\ref{theo1}}

\begin{proof}[Proof of Theorem \ref{theo1}]
Throughout the text it will be assumed
that we are rely on the principal dynamic picture
described in Lemma~\ref{lefolds} avoiding multiple references to it without special need.

Possessing asymptotical character and basing on the idea
of intermediate value theorem
Proposition~1 in \cite{AN}  predicts existing of an intersection point
of a trajectory with~$\partial(S)$
but does not guarantee its uniqueness, in other words,
can not answer the question whether trajectories intersect or touch  $\partial(S)$  more than once.
To answer this question we have  to analyze  the sign of the
inner product~$(\bold{V}, \nabla \gamma_i)$ on each $\Gamma_i$.
But $\operatorname{sign}\big(\bold{V}, \nabla \gamma_i\big)=\operatorname{sign}\big(G(\nu)a-F(\nu)\big)$
for all $\nu>1$ and $t>0$ independently of $i$ as it follows from \eqref{expres_W}.  Hence
it suffices  to determine the sign of the expression $G(\nu)a-F(\nu)$ or equivalently the sign of $a-F(\nu)$ (since  $G(\nu)>0$ for all $\nu>1$ by Lemma~\ref{function_F}).

\smallskip

$(1)$ {\it  The case $a\in (0,3/14)$}.
We know that  trajectories are moving from
$S\setminus I$  into $\operatorname{ext(S)}$ for all $\frac{1-2a}{4a}>\frac{1}{2}$.
Our goal is to show that
every trajectory
of \eqref{three_equat} originated in $\overline{S}\setminus I'$ can intersect  $\partial(S)$
at most once and remains in $\operatorname{ext}(S)$ forever and
no trajectory of \eqref{three_equat} initiated in~$\operatorname{ext}(S)$ can come into the domain $S$.

We claim  that  $(\bold{V},  \nabla \gamma_i)<0$  on~$\Gamma_i$, where $i$ is any of $1,2$ or $3$.
Indeed  $a-F(\nu)< 0$ in~\eqref{expres_W} for all $\nu>1$.
Suppose by contrary that there exists a value
$\nu^{\ast}>1$ such that $a-F(\nu^{\ast})\ge 0$.
Then   it should be $F(\nu^{\ast})\le a<3/14$ contradicting  $\min_{\nu\ge 1} F(\nu)=3/14$
proved in Lemma~\ref{function_F}.
Therefore, $(\bold{V},  \nabla \gamma_i)< 0$ on~$\Gamma_i$.
This means that for all $0<a<\frac{3}{14}$
the vector field $V$ of system \eqref{three_equat}
and the normal vector $\nabla \gamma_i$ of the conic surface $\Gamma_i$
form an obtuse angle on every point of $\Gamma_i$.
Recall that the normal vector $\nabla \gamma_i$  is directed inside the domain~$S$ on the part of its boundary $\Gamma_i\subset \partial(S)$ according to  Lemma~\ref{normal_inside}.
These facts mean that the vector field $V$
is directed outside the domain~$S$ at every point of the cone $\Gamma_i$
being the part of its boundary $\partial(S)$.
The same property can be proved for other components $\Gamma_j$ and $\Gamma_k$ of the set $\partial(S)$.

Trajectories of system \eqref{three_equat} initiated in $S\setminus I$
will  cross the border $\partial(S)$ sooner or later according to Proposition~1 in~\cite{AN}.
All such trajectories
never return back to  $\overline{S}$  due to
$(\bold{V},  \nabla \gamma_i)< 0$ proved above for each~$i=1,2,3$.
By the same reason trajectories
initiated in $\operatorname{ext}(S)$ never can get $\overline{S}$.
Trajectories started from $\partial(S)$  remain in  $\operatorname{ext}(S)$
obviously.

In final, consider trajectories of system \eqref{three_equat} initiated
in $I\cap S$. Some of them can leave $S$.
Indeed for all $a\in (0,1/4)$ and $i=1,2,3$
the invariant set $I_i\setminus I_i'$ is the stable (attracting) manifold for the singular point~$\bold{o_i}$
by Lemma~\ref{lefolds}.
Therefore for each $i=1,2,3$
every trajectory of~\eqref{three_equat} initiated in $(I_i\setminus I_i')\cap S$
will leave~$S$ along $(I_i\setminus I_i')\cap S$
attracting by~$\bold{o_i}$ which is located outside of  $S$ by
Lemma~\ref{sing_points_out_S} and then crossing  $\Gamma_i\subset \partial(S)$
 by Lemma~\ref{I_i_curves}.
As for trajectories started in $I'\cap S$
they will remain in~$S$  by Lemma~\ref{I_i_curves}.

\smallskip

$(2)$ {\it The case $a=3/14$}. Note that
 $(\bold{V},  \nabla \gamma_i)\le 0$ on~$\Gamma_i$.
More precisely,
the equality $(\bold{V},  \nabla \gamma_i)=0$
equivalent to
$3/14-F(\nu)=0$ can be satisfied on~$\Gamma_i$
at the unique point $\nu=4/3$ for each $i$.
This yields $\mu=1-\nu+2\sqrt{\nu(\nu-1)}=1$ (see~\eqref{numu}).
Hence  $(\bold{V},  \nabla \gamma_i)$ could be equal to $0$
along the straight line $x_i=4t/3$, $x_j=x_k=t$, $t>0$, only.
By Lemma~\ref{I_i_curves} on the surface $\Sigma$ defined by the equation
$x_1x_2x_3=c$ the invariant curve $I_i$ meets $\Gamma_i$ at the unique point
$x_i=cp_0^{-2}$, $x_j=x_k=p_0=\sqrt[3]{6c}/2$.
Since  $cp_0^{-2}=4p_0/3$  that common point is the singular point $\bold{o_i}$ with coordinates
$x_i=q\kappa$, $x_j=x_k=q$,
where $\kappa=(1-2a)/(2a)=4/3$   and $q=\sqrt[3]{c\kappa^{-1}}=p_0$.
Hence by uniqueness no trajectory of~\eqref{three_equat}
can cross the cone $\Gamma_i$ through  the straight line $x_i=4t/3$, $x_j=x_k=t$.

Therefore, we can assume $(\bold{V},  \nabla \gamma_i)<0$ on~$\Gamma_i$
and hence  for trajectories of~\eqref{three_equat}
initiated in
$S\setminus I$, $\operatorname{ext}(S)$ and $\partial(S)$
the similar dynamics are expected as in the previous case $a\in (0,3/14)$.
But now every trajectory initiated in $I\cap S$
will stay in $I\cap S$
because by Lemma~\ref{sing_points_out_S} there is no singular point
of~\eqref{three_equat} outside $S$  which can be attract trajectories
from~$(I\setminus I')\cap S$ and
trajectories started in $I'\cap S$ are still in~$S$
by Lemma~\ref{I_i_curves}.

\smallskip

$(3)$ {\it  The case $a\in (3/14, 1/4)$.}
For~\eqref{three_equat} we claim the following:
some its trajectories started in $\operatorname{ext}(S)$
can  intersect  $\partial(S)$ twice getting into~$S$ and returning  back
to $\operatorname{ext}(S)$ in the sequel; all  trajectories  will leave $S$ once and remain in $\operatorname{ext}(S)$ forever starting in $S\setminus I$.
The inner product $(\bold{V},  \nabla \gamma_i)$ can admit both positive or negative signs on~$\Gamma_i$.
Indeed according to  Lemma~\ref{function_F}
we know that the equation $F(\nu)=a$ admits exactly two
distinct real roots $\nu_1=\nu_1(a)$ and $\nu_2=\nu_2(a)$
given by \eqref{roots_nu}.
 As a consequence  $F(\nu)>a$  for $\nu \in (1,\nu_1)\cup (\nu_2,+\infty)$ and
$F(\nu)<a$ for $\nu\in (\nu_1,\nu_2)$
and, hence, $(\bold{V},  \nabla \gamma_i)< 0$ if $\nu\in (1,\nu_1)\cup (\nu_2,+\infty)$
and $(\bold{V},  \nabla \gamma_i)> 0$ if $\nu\in (\nu_1,\nu_2)$.
Knowing $\nu_1(a)$ and $\nu_2(a)$
provides  an opportunity to evaluate  values  $\nu_1^{\ast}:=\nu_1(a^{\ast})$ and
$\nu_2^{\ast}:=\nu_2(a^{\ast})$ at a fixed $a=a^{\ast}\in (3/14, 1/4)$.
Clearly $a-F(\nu)=0$ and hence $(\bold{V},  \nabla \gamma_i)=0$
along each straight line~\eqref{numu}
on the cone $\Gamma_i$, if  $\nu=\nu_1^{\ast}$ or $\nu=\nu_2^{\ast}$.
Therefore some trajectories of the system \eqref{three_equat} are coming into the domain $S$ from $\overline{\operatorname{ext}(S)}$
across the sector
$
\left\{x_i=\nu t, x_j=\mu t, x_k=t ~ | ~ t>0, \, \nu \in \big(\nu_1^{\ast},\nu_2^{\ast}\big) \right\}
$
of each conic surface $\Gamma_i$, $i=1,2,3$.
In sequel all these \glqq guest`` trajectories and trajectories started in $S\setminus I$
must leave $S$ according to Proposition~1 in~\cite{AN}.
Now we know that this happens through the sector
$
\left\{x_i=\nu t, x_j=\mu t, x_k=t ~ | ~ t>0, \, \nu \in \big(1, \nu_1^{\ast}\big) \cup \big(\nu_2^{\ast}, +\infty\big) \right\}
$
of each $\Gamma_i$, where $i=1,2,3$.

As for trajectories started in $I\cap S$ they
can not leave $I\cap S$ by Lemma~\ref{sing_points_out_S}.

\smallskip

$(4)$ {\it The case $a\in (1/4, 1/2)$}.
According to Proposition~1 in~\cite{AN}
trajectories of~\eqref{three_equat} originated in $\operatorname{ext}(S)\setminus I$ gets into $S$
in finite time in the case $\frac{1-2a}{4a}<\frac{1}{2}$.
We have also
trajectories starting in $I\cap \operatorname{ext}(S)$.
They reach~$S$ along $I\cap \operatorname{ext}(S)$
attracting by  the stable manifold of the singular point $(1,1,1)$
and crossing $\partial(S)$ according to  Lemma~\ref{I_i_curves}.

It is obvious that $(\bold{V},  \nabla \gamma_i)>0$ on   $\Gamma_i$
since  $a-F(\nu)> 1/4-F(\nu)>0$ for $\nu>1$ by Lemma~\ref{function_F}.
This means the associated vector field $V=(f,g,h)$
is directed inside the domain~$S$ on each $\Gamma_i$
and hence trajectories never can leave $S$ visiting it once or starting in it.

\smallskip

$(5)$ {\it  The case $a=1/4$}.
Clearly
$a-F(\nu)= 1/4-F(\nu)>0$ for all $\nu>1$ at $a=1/4$.
Therefore all trajectories will stay in  $S$ if they originated in $\overline{S}$.
But now Proposition~1 is useless to predict the behavior of trajectories originated in $\operatorname{ext}(S)$ since
neither $\frac{1-2a}{4a}<\frac{1}{2}$  nor $\frac{1-2a}{4a}>\frac{1}{2}$
can be satisfied by~$a=1/4$.
The value $a=1/4$  is also a bifurcation value.
Four non degenerate singular points~\eqref{ququ}
of system~\eqref{three_equat} merge to the unique  degenerate singular point $(1,1,1)$.
We will use results of~\cite{AANS1} where a planar analysis
was given in the $(x_1,x_2)$-coordinate plane
(note that phase portraits in the coordinate planes $(x_2, x_3)$ and $(x_1,x_3)$  will be the same due to symmetry in~\eqref{three_equat}).
As it was proved in Theorem~2 in~\cite{AANS1}
the system~\eqref{rnrf_sc} admits
the unique  singular point~$(1,1)$ at $a=1/4$
being a linear zero type saddle
with $6$ hyperbolic sectors in its neighborhood
according to  a classification given in~\cite{JiangLlibre}.
These sectors are separated by unstable (repelling) manifolds $c_i'$
and stable (attracting) manifolds $c_i\setminus c_i'$ of $(1,1)$
(the curves $c_i$ and $c_i'$ was introduced in Corollary~\ref{c_i_curves}),
see Figure~\ref{Gammas4}.
Due to these facts  there exist trajectories of~\eqref{rnrf_sc}
that will tend to  $(1,1)\in D$ and hence could reach  $D$
if they started  at least in the stable manifolds $c_i\setminus c_i'$
or in points sufficiently close to them.
This  establishes that some trajectories of  \eqref{three_equat}
originated in $\operatorname{ext}(S)$
can intersect  $\Gamma_2\subset \partial(S)$ and get into~$S$.
Theorem \ref{theo1} is proved.
\end{proof}

\begin{figure}[h]
\centering
\includegraphics[angle=0, width=0.45\textwidth]{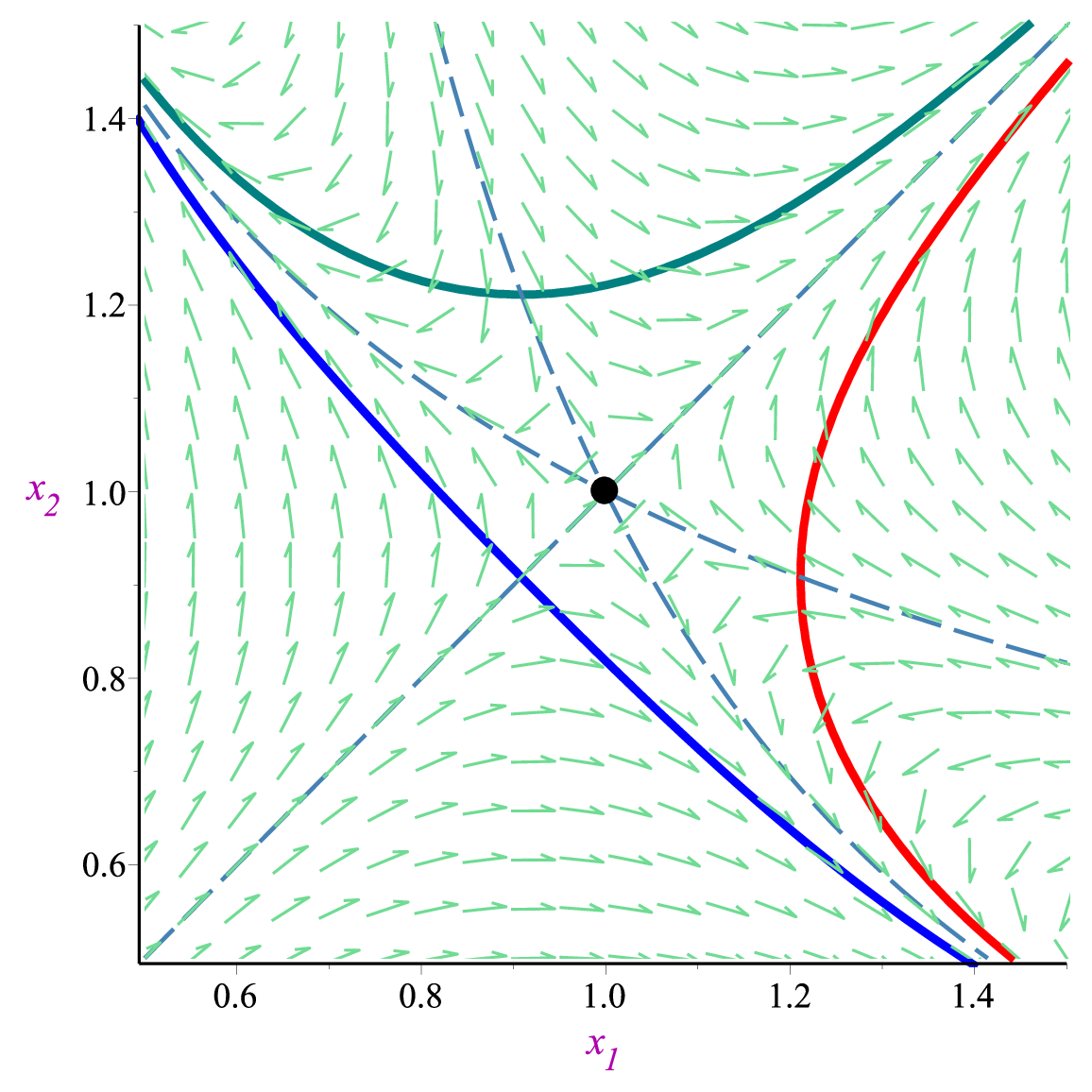}
\caption{The curves $s_1,s_2$ and $s_3$ (in red, teal and blue colors), the invariant curves $c_1,c_2$ and $c_3$ (dash lines) and the phase portrait  of system~\eqref{rnrf_sc} corresponding to $a=1/4$.}
\label{Gammas4}
\end{figure}

\section{Planar illustration of results and some supplementary discussions}

In Remark~\ref{refolds} we noted that surfaces $x_1x_2x_3=c$
are most significant for study the dynamics of the system~\eqref{three_equat}
because reflect leading motions of trajectories at a fixed $c>0$.
Actually studying could be reduced to the surface  $x_1x_2x_3=1$
since $f_i$ are all homogeneous and~\eqref{three_equat} is autonomous.
Replacing  $x_i=X_i\sqrt[3]{c}$, $t=\tau\sqrt[3]{c}$,
we get a new system in the new variables but of the same form as the initial~\eqref{three_equat}.
Since $c>0$ such a passing to the new variables does not change the orientation of trajectories,
and stability (or nonstability) of manifolds of singular points will be preserved as well.
Thus let us  consider the invariant surface  $x_1x_2x_3=1$ without loss of generality.
The corresponding planar system obtained on it and projected to the plane $x_3=0$ is
the following:
\begin{equation}\label{rnrf_sc}
\aligned
\dot{x}_1(t)&= &\widetilde{f}(x_1,x_2):=x_1x_2^{-1}+x_1^2x_2 - 2ax_1\left(2x_1^2-x_2^2-x_1^{-2}x_2^{-2}\right)-2,                \\
\dot{x}_2(t)&= &\widetilde{g}(x_1,x_2):=x_2x_1^{-1}+x_1x_2^2- 2ax_2\left(2x_2^2-x_1^2-x_1^{-2}x_2^{-2}\right)-2,
\endaligned
\end{equation}
where\,
$\widetilde f(x_1,x_2)\equiv  f_1\big(x_1,x_2, (x_1x_2)^{-1}\big)$, ~
$\widetilde g(x_1,x_2)\equiv  f_2\big(x_1,x_2, (x_1x_2)^{-1}\big)$.

Let us consider  \eqref{rnrf_sc} for
various values of the parameter $a\in (0,1/2)$ in the $(x_1,x_2)$-coordinate plane
to  give a more visible illustration of  Theorem \ref{theo1}.
Recall some well-known concepts from
the theory of dynamical systems (see~\cite{Dumort} for example). The point $(x_0,y_0)$ is said to be a singular (equilibrium) point of system~\eqref{rnrf_sc}
if the following equalities hold: $\widetilde{f}(x_0,y_0)=0$ and $\widetilde{g}(x_0,y_0)=0$.
Let $\widetilde{\bold{J}}:=\widetilde{\bold{J}}(x_0,y_0)$ be a matrix of linear parts of~\eqref{rnrf_sc}.
Introduce values
$\delta:=\operatorname{det}\big(\widetilde{\bold{J}}\big)$, \, $\rho:=\operatorname{trace}\big(\widetilde{\bold{J}}\big)$, \,
$\sigma:=\rho^2-4\delta$.

It follows  from~\eqref{ququ} in the case $a\in (0,1/2)\setminus \{1/4\}$
that~\eqref{rnrf_sc} admits only  four
singular points  $(1,1)$, $(q\kappa, q)$,
$(q,q\kappa)$ and $(q,q)$
being  orthogonal projections of~\eqref{ququ}
onto the plane $x_3=0$.
We can easily evaluate that
$$
\delta=(4a-1)^2>0, \quad \rho=-2(4a-1), \quad \sigma=\rho^2-4\delta=0
$$
for $(1,1)$ and
$$
\delta=\frac{(2a+1)(4a-1)^2}{(2a-1)q^2}<0
$$
for $(q\kappa, q)$,
$(q,q\kappa)$ and $(q,q)$.
Thus $(1,1)$ is a hyperbolic node and
$(q\kappa, q)$,
$(q,q\kappa)$, $(q,q)$ are hyperbolic saddles of~\eqref{rnrf_sc}
according to~\cite{Dumort}.
Note also that the node $(1,1)$  is unstable (repelling) if $a\in(0,1/4)$
and stable (attracting) if $a\in (1/4, 1/2)$
due to the signs of $\rho$.
Since each singular point of~\eqref{rnrf_sc}
are hyperbolic at $a\in (0,1/2)\setminus \{1/4\}$
then according to Hartman-Grobman theorem
nonlinear system~\eqref{rnrf_sc}
has the phase portrait topologically  equivalent to the phase portrait
of the corresponding linearized system  in a sufficiently small neighborhood
of that point.
The case $a=1/4$ leading to $\delta=\rho=\sigma=0$ is much more complicated.
An appropriate material can be found in  \cite{Ab1, AANS1, AANS2} and references therein.

\smallskip

{\it The case $a\in (0,3/14)$}.
The behavior of trajectories of~\eqref{rnrf_sc} can  be explained
by the influence of attracting and repelling  manifolds of
the singular points of~\eqref{rnrf_sc}.
Actually the unstable node  $(1,1)\in D$ repels trajectories
in $D$ whereas the saddles $(q\kappa, q)$,
$(q,q\kappa)$ and $(q,q)$  attract  them
lying outside $D$ by Lemma~\ref{sing_points_out_S}
(the saddles are depicted in the left panel of Figure~\ref{Gammas2}
by points in red, teal and blue colors respectively).
This explains why trajectories are leaving the domain~$D$.
The repelling manifold of   $(1,1)$
is tangent to the repelling manifold
$\operatorname{Span}\left\{(1,0),(0,1)\right\}$
of the linearized system of~\eqref{rnrf_sc}
which corresponds to  eigenvalues $\lambda_1=\lambda_2=1-4a>0$ 
of the matrix~$\widetilde{\bold{J}}(1,1)$.
As for saddles it suffices describe manifolds of $(q,q)$ only.
It has attracting
and repelling manifolds tangent to  $\operatorname{Span}\left\{(1,1)\right\}$
and $\operatorname{Span}\left\{(-1,1)\right\}$ respectively which correspond
to eigenvalues $\lambda_1=\frac{(1-4a)(2a-1)}{2a}<0$ and  $\lambda_2=\frac{(1-4a)(2a+1)}{2a}>0$
of the matrix $\widetilde{\bold{J}}(q,q)$.

\begin{figure}[h]
\centering
\includegraphics[angle=0, width=0.45\textwidth]{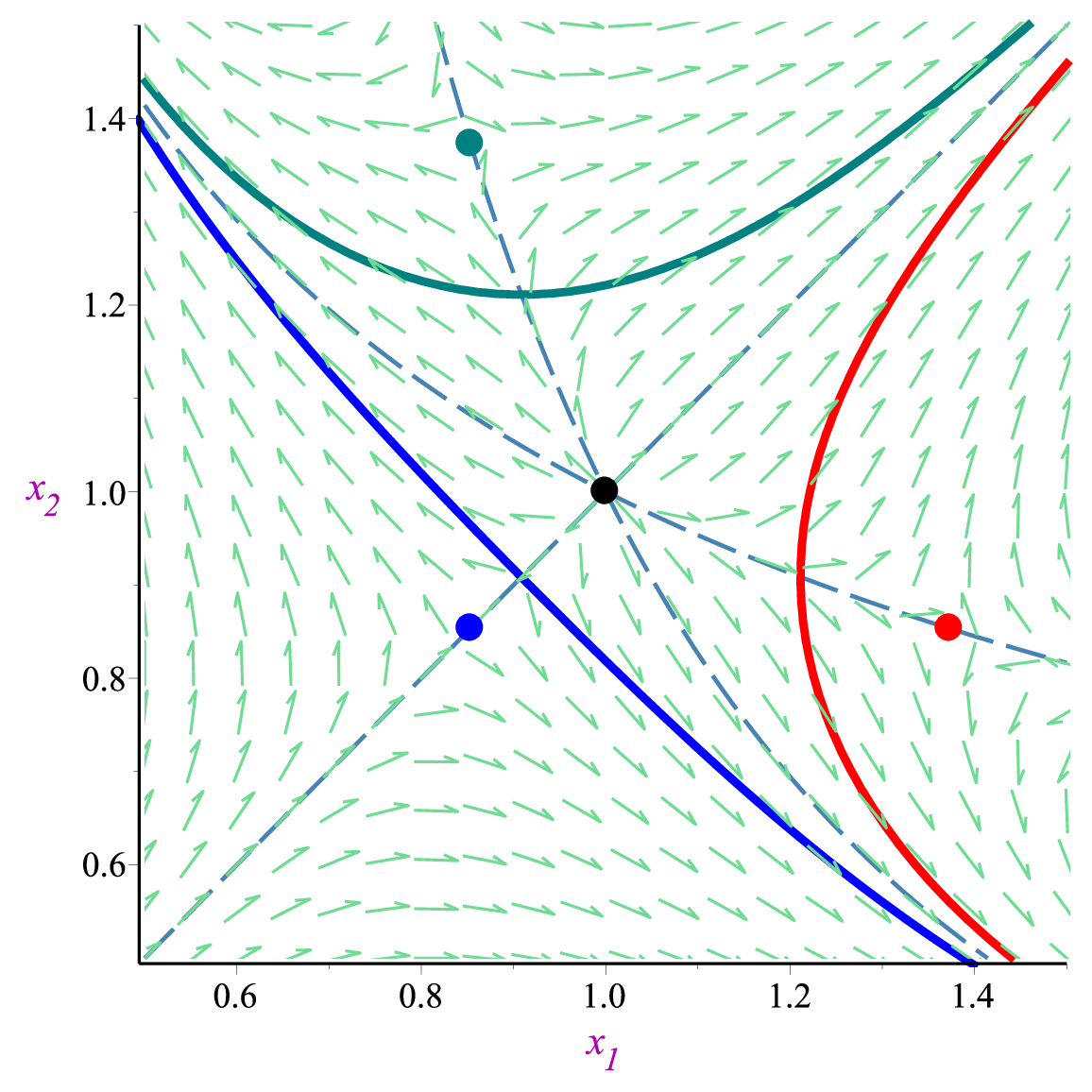}
\includegraphics[angle=0, width=0.45\textwidth]{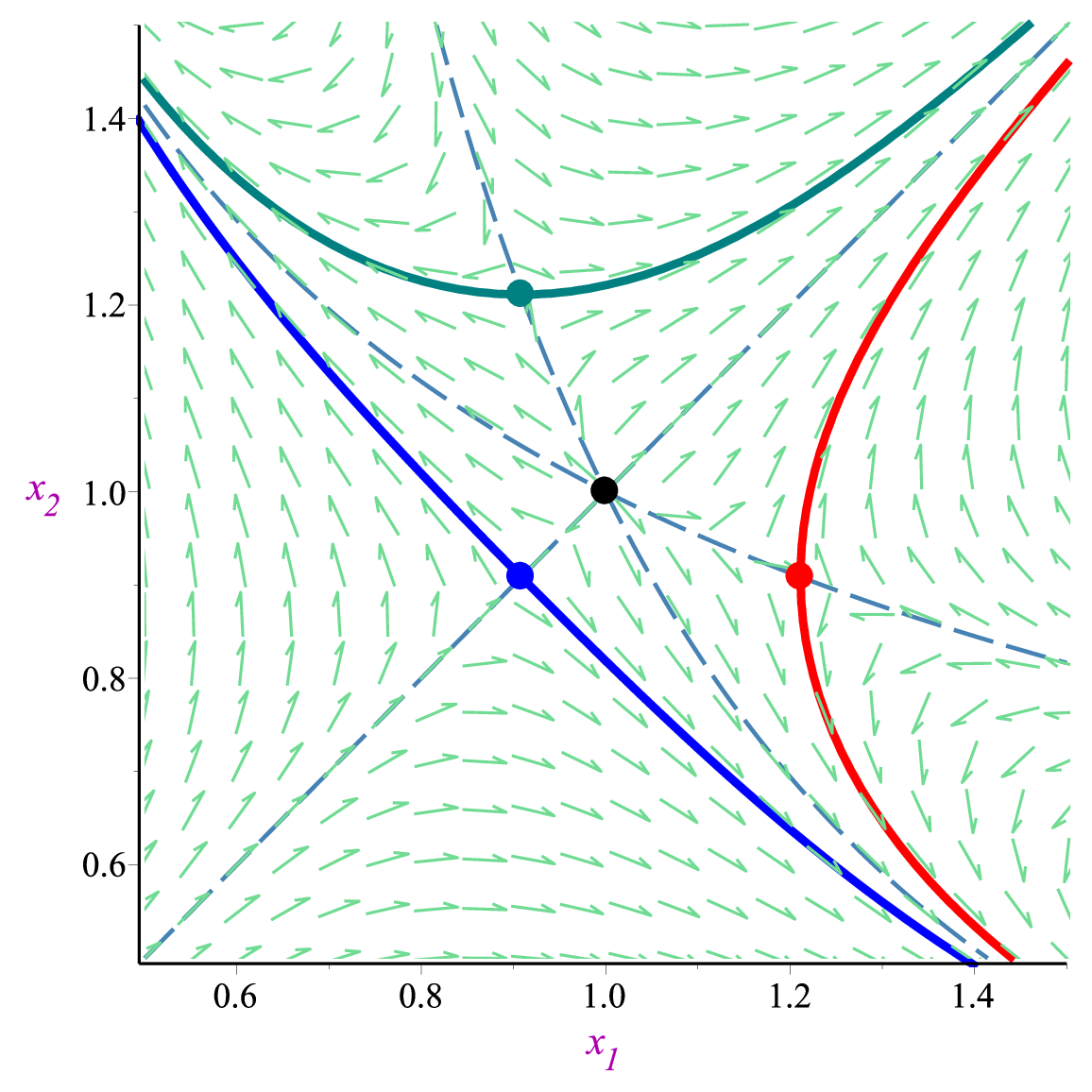}
\caption{Phase portrait  of system~\eqref{rnrf_sc}
at $a\in (0,3/14)$ (the left panel) and $a=3/14$ (the right panel).}
\label{Gammas2}
\end{figure}

\smallskip

{\it The case $a=3/14$}.
All singular points of system~\eqref{rnrf_sc}
preserve their stable and unstable manifolds found in the previous case~$a\in (0,3/14)$
and hence the topological structure of phase portrait will be the same,
but now saddles are all located in the border  $\partial(D)$ of the domain $D$
due to Lemma~\ref{sing_points_out_S}.
This circumstance can lead to a \glqq weaker`` influence of saddles
and as a consequence trajectories may get $D$ across $\partial (D)$
being  almost parallel to $\partial (D)$.
Consider, for example, $s_1\subset \partial (D)$
(see the red curve on the right  panel of Figure~\ref{Gammas2}).
The curve $s_1$ can  be parameterized  as the following
\begin{eqnarray}\label{param}
x_1(t) =t\, \left(\frac{t-1+2\sqrt{t(t-1)}}{(3t+1)(t-1)t}\right)^{\frac{1}{3}}, \quad
x_2(t) =\left(\frac{t-1+2\sqrt{t(t-1)}}{(3t+1)(t-1)t}\right)^{\frac{1}{3}}, \quad t>1,
\end{eqnarray}
substituting $x_1=tx_2$,  $x_2=\sqrt[3]{u}$ into $l_1=0$ in \eqref{L_i}
and solving it with respect to $u$.
Note that
$$
\lim_{t\to 1+0}x_1(t)=\lim_{t\to 1+0}x_2(t)=+\infty, \quad
\lim_{t\to +\infty}x_1(t)=+\infty, \quad \lim_{t\to +\infty}x_2(t)=0.
$$

The following asymptotic representations can be useful
for quick  estimations:
$$
x_1(t)\sim x_2(t) \sim 2^{-1/3}(t-1)^{-1/6}  \mbox{~~as~~} t\to 1+0 \mbox{~~and~~}
 x_1(t) \sim t^{1/3}, ~ x_2(t) \sim t^{-2/3} \mbox{~~as~~} t\to +\infty.
$$

On the curve $s_1$ introduce the function
$$
\Delta(t):=\dfrac{\dot{x_2}(t)}{\dot{x_1}(t)}-\dfrac{\widetilde{g}(x_1(t),x_2(t))}{\widetilde{f}(x_1(t),x_2(t))}
$$
that means the difference between
the slope of the tangent vector
$\bold {\tau}(t)=\big(\dot{x_1}(t),\dot{x_2}(t)\big)$ of $s_1$ and the slope of the vector field
$
\widetilde{\bold{V}}(t):=\Big(\widetilde{f}\big(x_1(t),x_2(t)\big),\widetilde{g}\big(x_1(t),x_2(t)\big)\Big)
$
associated with system~\eqref{rnrf_sc},
where $x_1(t)$ and $x_2(t)$ are determined as in \eqref{param}.
As calculations show $\Delta(t)>0$ for $1<t<4/3$ and
$\Delta(t)<0$ for $t>4/3$.
At $t=4/3$ the function $\Delta(t)$ is not defined since we have
the singular point $(q\kappa,q)$
exactly on $s_1$ (the red point on~$s_1$), and hence
$\dfrac{\widetilde{g}}{\widetilde{f}}=\dfrac{0}{0}$.
Thus $\operatorname{slope}(\widetilde{\bold{V}})<\operatorname{slope}(\bold{\tau})$
at $1<t<4/3$ and $\operatorname{slope}(\widetilde{\bold{V}})>\operatorname{slope}(\bold{\tau})$
at $t>4/3$  confirming assertions of Theorem \ref{theo1}.
But our  interest is not only the sign of the difference
of slopes but its quantitative estimation too.
Consider the angle
$
\alpha(t):=\arccos\, \dfrac{\widetilde{V}_1\tau_1+\widetilde{V}_2\tau_2}{\|\widetilde{V}\|\,\|\tau\|}
$
between $\widetilde{V}$ and $\tau$.
As calculations show
$$
3.1125\le \alpha(t)<\pi
$$
for $1<t<4/3$ (see the left panel of Figure~\ref{ugol}).
In addition
$
\lim_{t\to 1+0}\alpha(t)=\lim_{t\to 4/3-0}\alpha(t)=\pi
$.
This means that trajectories are almost parallel to the tangent of $s_1$
and intersect~$s_1$ under angles in a narrow range $[178.33^{\circ}, 180^{\circ})$ given in degrees.
For $t>1$ we have
$$
0<\alpha(t)\le 0.211  ~~~ (\mbox{in degrees~~} 0^{\circ}<\alpha\le  1.21^{\circ})
$$
 with
$\lim_{t\to 4/3+0}\alpha(t)=\lim_{t\to +\infty}\alpha(t)=0$
(see the right panel of Figure~\ref{ugol}).
In any case  at every point of the curve $s_1$
trajectories are \glqq almost parallel`` to $s_1$ (to its tangent),
nevertheless they intersect~$s_1$ under small angles  not exceeding
$2^{\circ}$.

\begin{figure}[h]
\centering
\includegraphics[angle=0, width=0.45\textwidth]{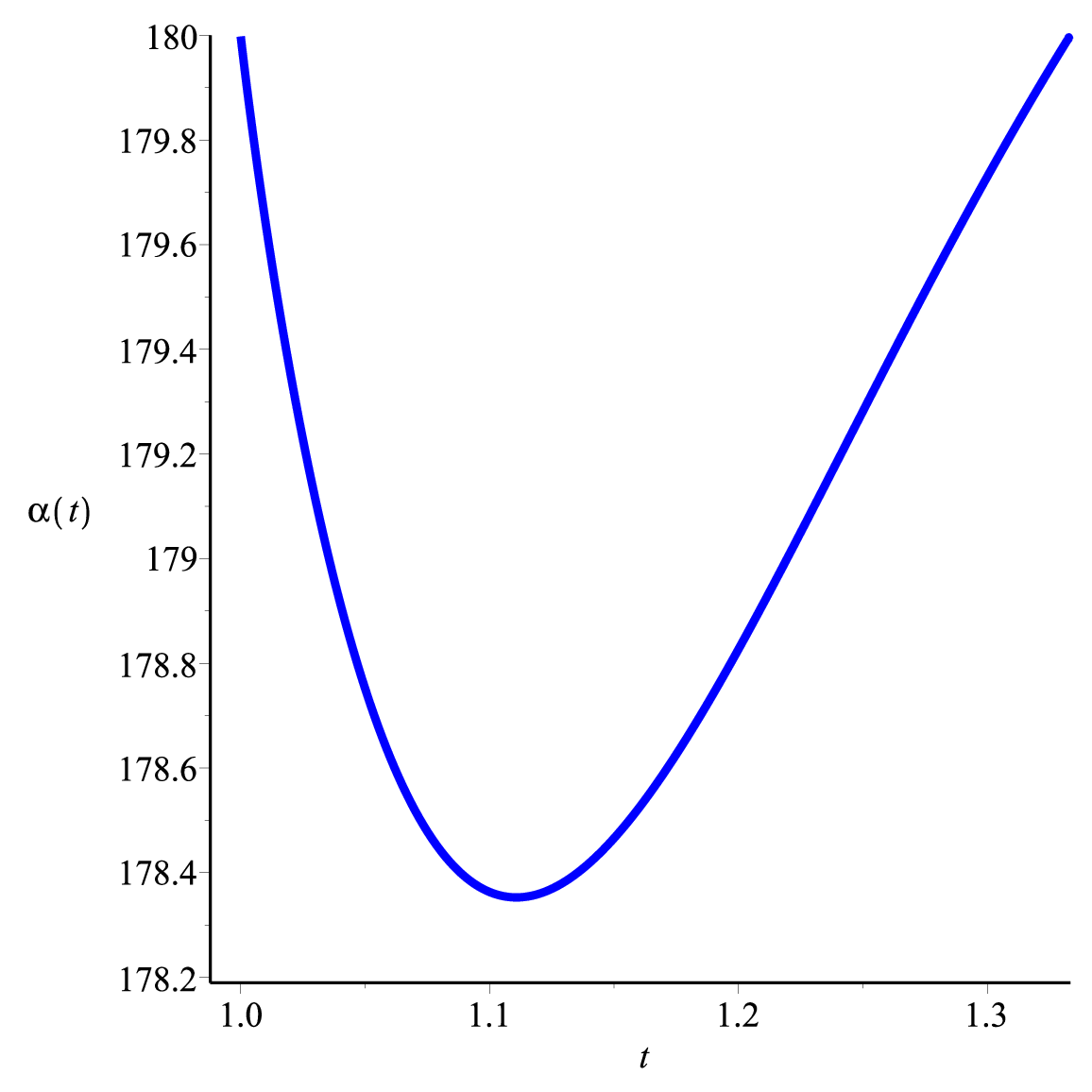}
\includegraphics[angle=0, width=0.45\textwidth]{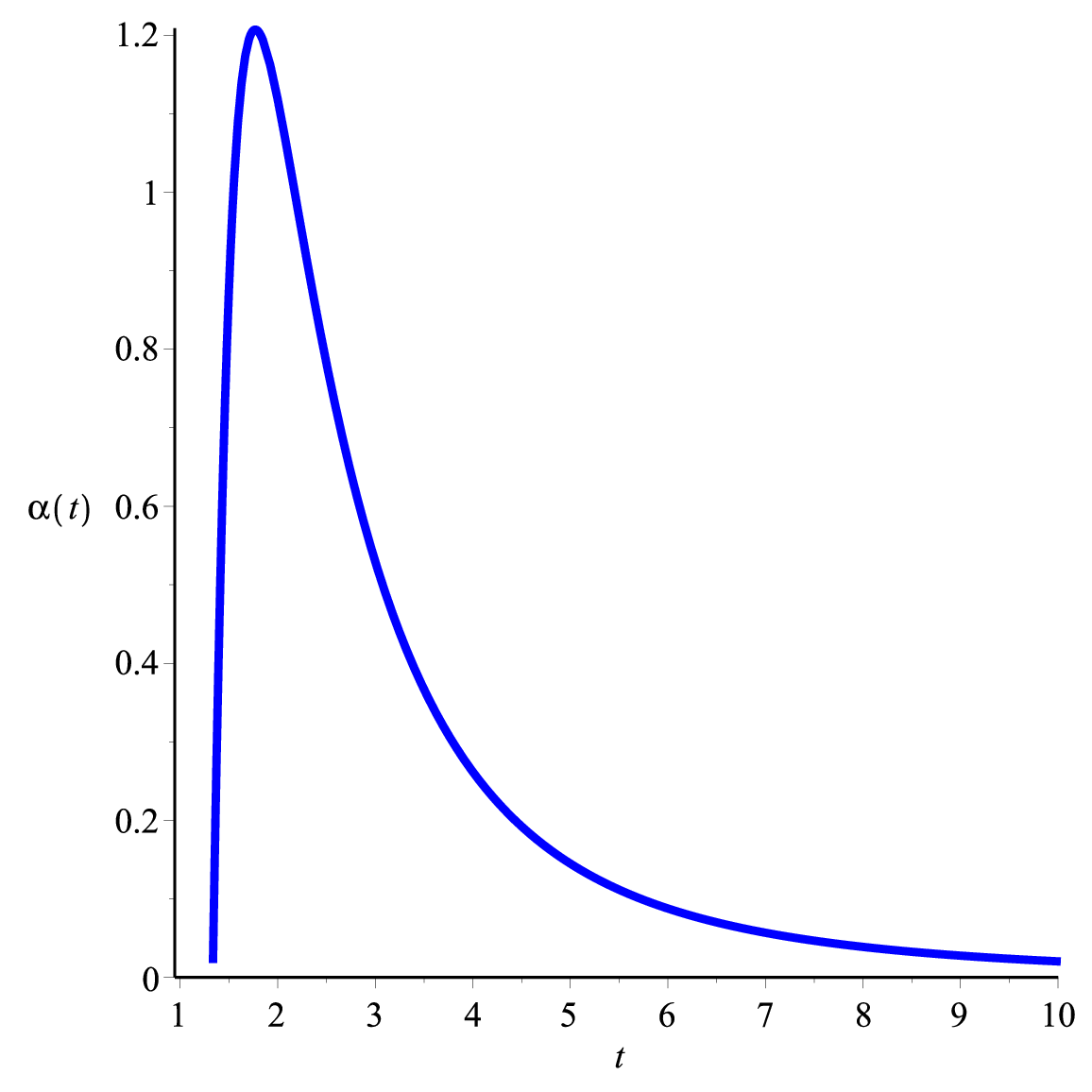}
\caption{Angle $\alpha(t)$ measured in degrees in the case $a=3/14$: for $1<t<4/3$ (the left panel);
for $t>4/3$ (the right panel).}
\label{ugol}
\end{figure}

\smallskip

{\it The case $a\in (3/14, 1/4)$}.
A planar illustration is shown in the left panel of Figure~\ref{Gammas3}
corresponding to the value $a=a^{\ast}=\frac{3}{14}+0.006$.
By Theorem~1 there exist exactly two points, say $P'$ and $Q'$, on~$\Sigma \cap \Gamma_i$
at which $(\bold{V},  \nabla \gamma_i)=0$.
Coordinates of these points can be found  by~\eqref{numu}
at corresponding values of $\nu_1$ and $\nu_2$ defined by~\eqref{roots_nu}.
Take $i=2$ for clarity.
Then on the plane $x_3=0$ trajectories of \eqref{rnrf_sc}
can be  tangent to the boundary $s_2$ of $D$ exactly at two points
being the projections of $P'$ and $Q'$ onto $x_3=0$,
denote them by $P$ and $Q$.

It follows from~\eqref{roots_nu} that
$\nu_1^{\ast}=\nu_1(a^{\ast})\approx 1.0793$,
$\mu_1^{\ast}=\mu(\nu_1^{\ast})\approx 0.5059$.
The fixed value
of the parameter~$t$ in~\eqref{numu} will be actualized
as $t_1^{\ast}=\sqrt[3]{(\nu_1^{\ast}\mu_1^{\ast})^{-1}}\approx 1.2234$.
This implies that the coordinates of the point~$P$ are
$$
x_1^{P}=t_1^{\ast}\approx 1.2234, \quad x_2^{P}=\nu_1^{\ast} t_1^{\ast}\approx 1.3204.
$$

By the same way we can find
coordinates $x_2^{\ast}$ and $y_2^{\ast}$ of $Q$.
We easily get
$\nu_2^{\ast}=\nu_2(a^{\ast})\approx 2.1332$,
$\mu_2^{\ast}=\mu(\nu_2^{\ast})\approx 1.9764$, $t_2^{\ast}=\sqrt[3]{(\nu_2^{\ast}\mu_2^{\ast})^{-1}}\approx 0.6190$.
Then
$$
x_1^{Q}=t_2^{\ast}\approx 0.6190, \quad
x_2^{Q}=\nu_2^{\ast} t_2^{\ast}\approx 1.3204.
$$

Denote now by $\operatorname{arc}(PQ)$ a part of the curve $s_2 \subset \partial (D)$
with endpoints $P$ and $Q$.
Then trajectories reach  $D$
through   $\operatorname{arc}(PQ)$ and
leave $D$ through  $s_2\setminus \operatorname{arc}(PQ)$.

\begin{figure}[h]
\centering
\includegraphics[angle=0, width=0.45\textwidth]{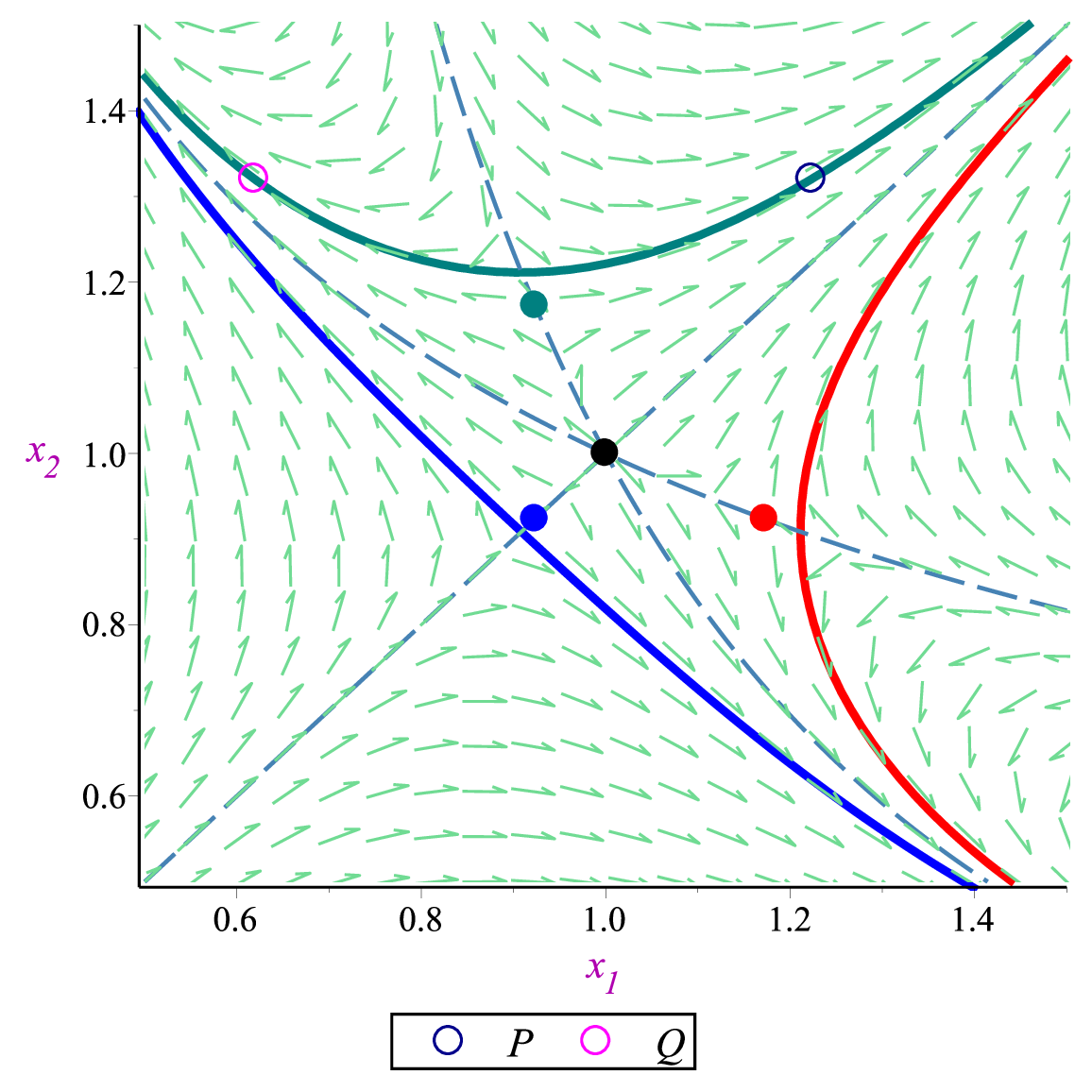}
\includegraphics[angle=0, width=0.45\textwidth]{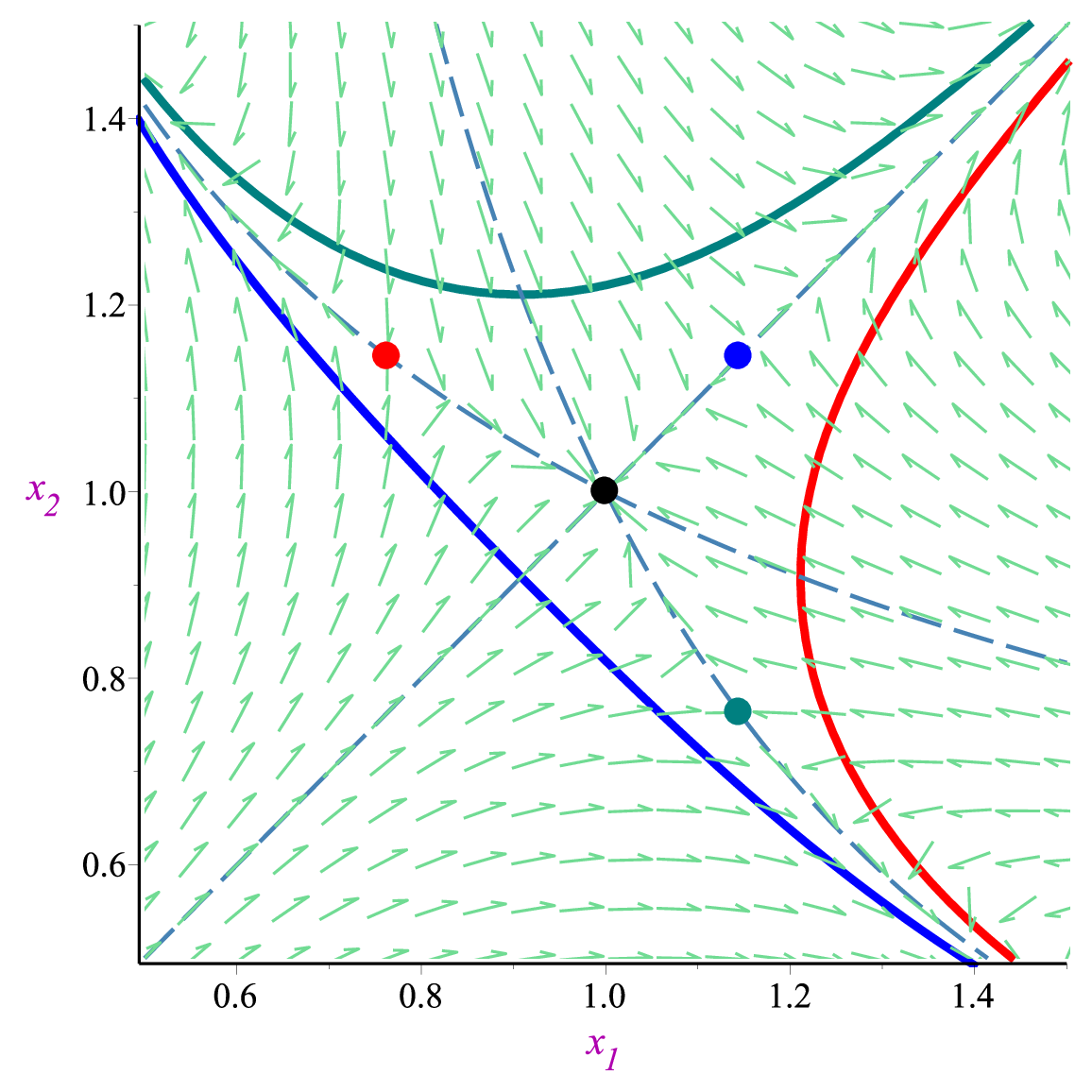}
\caption{Phase portrait  of system~\eqref{rnrf_sc}
at $a\in (3/14, 1/4)$ (the left panel) and $a\in (1/4,1/2)$ (the right panel).}
\label{Gammas3}
\end{figure}

\smallskip

{\it The case $a=1/4$}.
Introduce new variables $x=x_3/x_1$ and $y=x_3/x_2$.
Consider the connected component $\Omega$ of $D$ defined by the conditions
$0<x_2<x_1$ and $x_2<x_1^{-2}$.
Its  homeomorphic image~$\Omega'$ will be described by $y>x>1$
in the new coordinate plane $(x,y)$.
Observe that
the equation $l_3(x_1,x_2)=0$ of the boundary curve $s_3$ admits an explicit solution
$$
y_{b}=x\left(x-1+2\sqrt{x(x-1)}\right)(3x+1)^{-1}(x-1)^{-1}
$$
in new variables.
The part  $s_3\cap\Omega$ of  $s_3$ satisfying $x_2\to 0$ as $x_1\to +\infty$
is mapped onto a new curve~$s_3'$ in $\Omega'$ satisfying $y\to +\infty$ as $x\to 1+0$.
Moreover, the following asymptotic formula holds
$$
y_b\sim  2^{-1}\sqrt{(x-1)^{-1}} \mbox{~~as~~} x\to 1+0
$$
for points on $s_3'$.
The system \eqref{rnrf_sc} can be reduced to a new differential equation  $dy/dx=\widehat{g}(x,y)/\widehat{f}(x,y)$ in $\Omega'$.
Using  ideas of Lemma~2 in~\cite{AN} we obtain
$$
\frac{\widehat{g}}{\widehat{f}}\sim \frac{2a-1}{4a}\frac{y}{x-1}=-\frac{y}{2(x-1)}
\mbox{~~as~~} x\to 1+0, ~ y\to +\infty.
$$
For solutions $y=\phi(x,C)$ of $dy/dx=\widehat{g}(x,y)/\widehat{f}(x,y)$ it follows then
$$
y\sim  \sqrt{C(x-1)^{-1}} \mbox{~~as~~} x\to 1+0,
$$
where $C$ is any positive constant.
Hence for $x$ infinitely close to $1$  tangents of integral curves $y=\phi(x,C)$  could be parallel to corresponding
tangents of the border curve $s_3'$. Therefore some trajectories need not cross~$s_3'$.
This explains  why Proposition~1 of~\cite{AN} fails at $a=1/4$ and hence
not every trajectory could intersect~$\partial(S)$ as stated in Theorem \ref{theo1}.
For solutions of~$dx_2/dx_1=\widetilde{g}(x_1,x_2)/\widetilde{f}(x_1,x_2)$  considered in~$\Omega$
 we obtain
$$
x_1= \dfrac{1}{\sqrt{x_2}}-\dfrac{Cx_2^{5/2}}{2}+O\left(x_2^{11/2}\right) \mbox{~~ as~~} x_2\to +0,
$$
passing to the old variables,
while the corresponding asymptotic of the curve $s_3$ is
$$
x_1 = \dfrac{1}{\sqrt{x_2}}-\dfrac{x_2^{5/2}}{8}+O\left(x_2^{11/2}\right) \mbox{~~ as~~} x_2\to +0.
$$
Note also that $x_1=\dfrac{1}{\sqrt{x_2}}$ is the invariant curve $c_3$ exactly.

\smallskip

{\it The case $a\in (1/4, 1/2)$}.
A typical situation that confirms assertions of Theorem~\ref{theo1}
is illustrated in the right  panel of Figure~\ref{Gammas3}.

\smallskip
It should be noted that the finite time guarantied  for trajectories to get $\partial(D)$ may be as long as we want if values $a$ are close to $1/4$
(see Figures~\ref{RK1} and \ref{RK2}).
Let us take $a=a_0:=0.26$ and pass to the variables $x=x_3/x_1$ and $y=x_3/x_2$.
Choose an initial point $O$ with coordinates $x_0=1+10^{-3}$, $y_0=20$
(see the left panel of Figure~\ref{RK1}).
The unique exact solution $y=\phi(x)$ of the initial value problem
\begin{equation}\label{IV_problem}
\frac{dy}{dx}=\frac{\widehat{g}(x,y)}{\widehat{f}(x,y)}, \quad y(x_0)=y_0
\end{equation}
can be found approximately  by numerical methods
subdividing a given interval $[b, x_0]\subset [1,x_0]$  into~$N$ equal subsegments
and selecting mesh points $x_i=x_0+ih$ for  $i=0,\dots, N$,
where  $h=(b-x_0)/N$ is a negative step size.
The Runge~-- Kutta method applied for $b=1+{10}^{-6}$ and  $N=5000$
shows that  the integral curve $y=\phi(x)$  initiated in $O(x_0,y_0)$
attains the boundary $s_3'\cap \Omega'$ of $D'\cap \Omega'$ at a point
with coordinates
$$
(x,y)\approx (1.000008792, 168.88),
$$
where
the first negative difference $y_r[j]-y_b(x_j)<0$
happens for the mesh point $x_j=1.000008792$,  $j=4961$ and
$y_r[i]$ means an approximative value to $y=\phi(x_i)$ obtained
by Runge~-- Kutta method, $i=1,\dots, N$.
Denote this point by~$R$ (see the left panel of Figure~\ref{RK2}).

In  cases $x\to 1+0$ we can also suggest an asymptotic representation
$y_a=C\,(x-1)^{\frac{1}{2}-\frac{1}{4a}}$
of~$y=\phi(x)$. Substituting $a=a_0$ into $y_a(x)$ and using the initial condition $y_a(x_0)=y_0$
we determine a positive parameter~$C$ and obtain an actualized asymptotic
$$
y_a=0.8249252769\,(x-1)^{-0.4615384615}
$$
to the exact solution $y=\phi(x)$.
As an intersection of the graphs of $y_a(x)$  and $y_b(x)$ we obtain
a point $A$  with coordinates (see the right panel of Figure~\ref{RK2})
$$
(x,y)\approx (1.000002264, 332.55).
$$
Thus numeric experiments also confirm that a long time can be required for trajectories of~\eqref{rnrf_sc}
to intersect~$\partial (D)$ if $a\to 1/4$.

\begin{figure}[h]
\centering
\includegraphics[angle=0, width=0.45\textwidth]{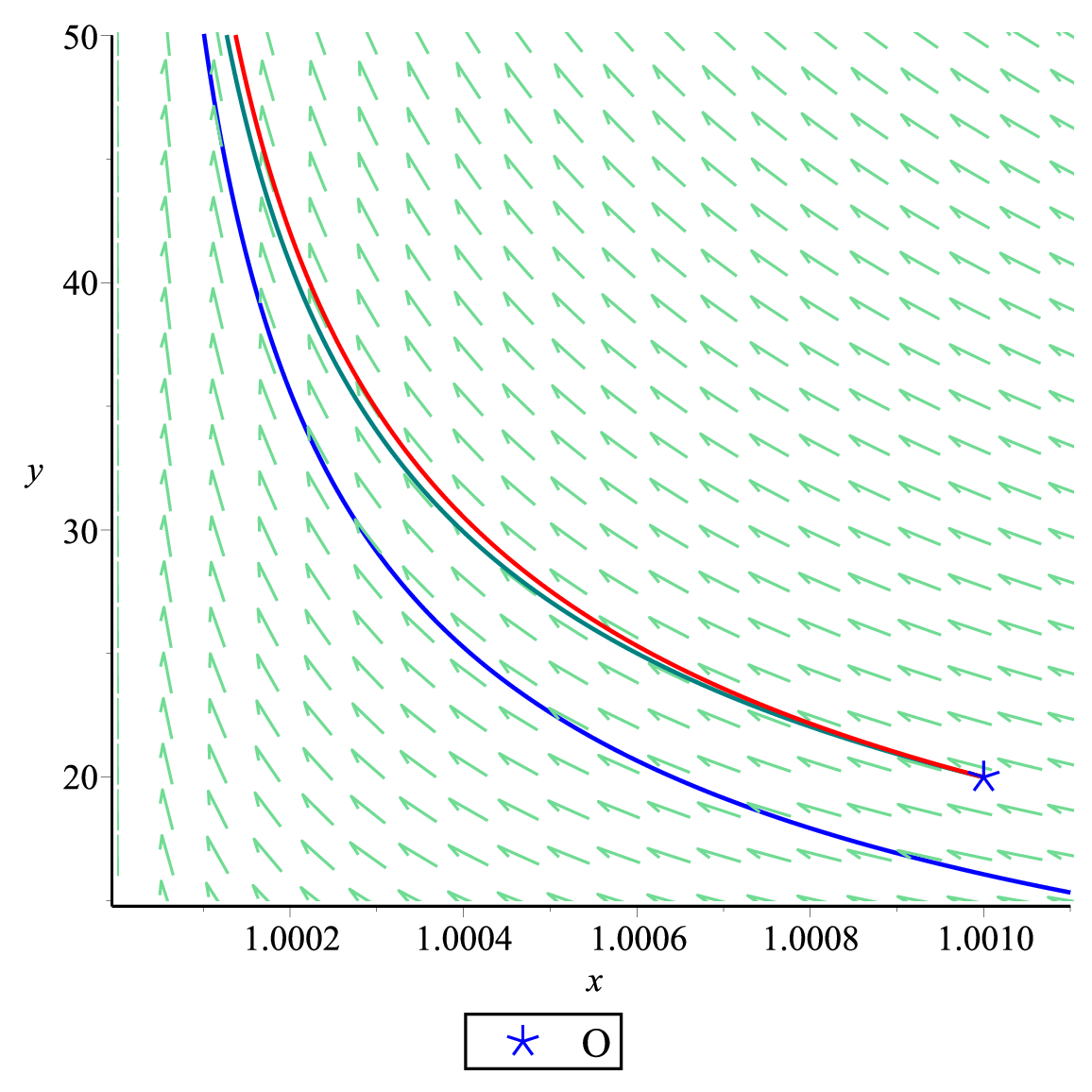}
\includegraphics[angle=0, width=0.45\textwidth]{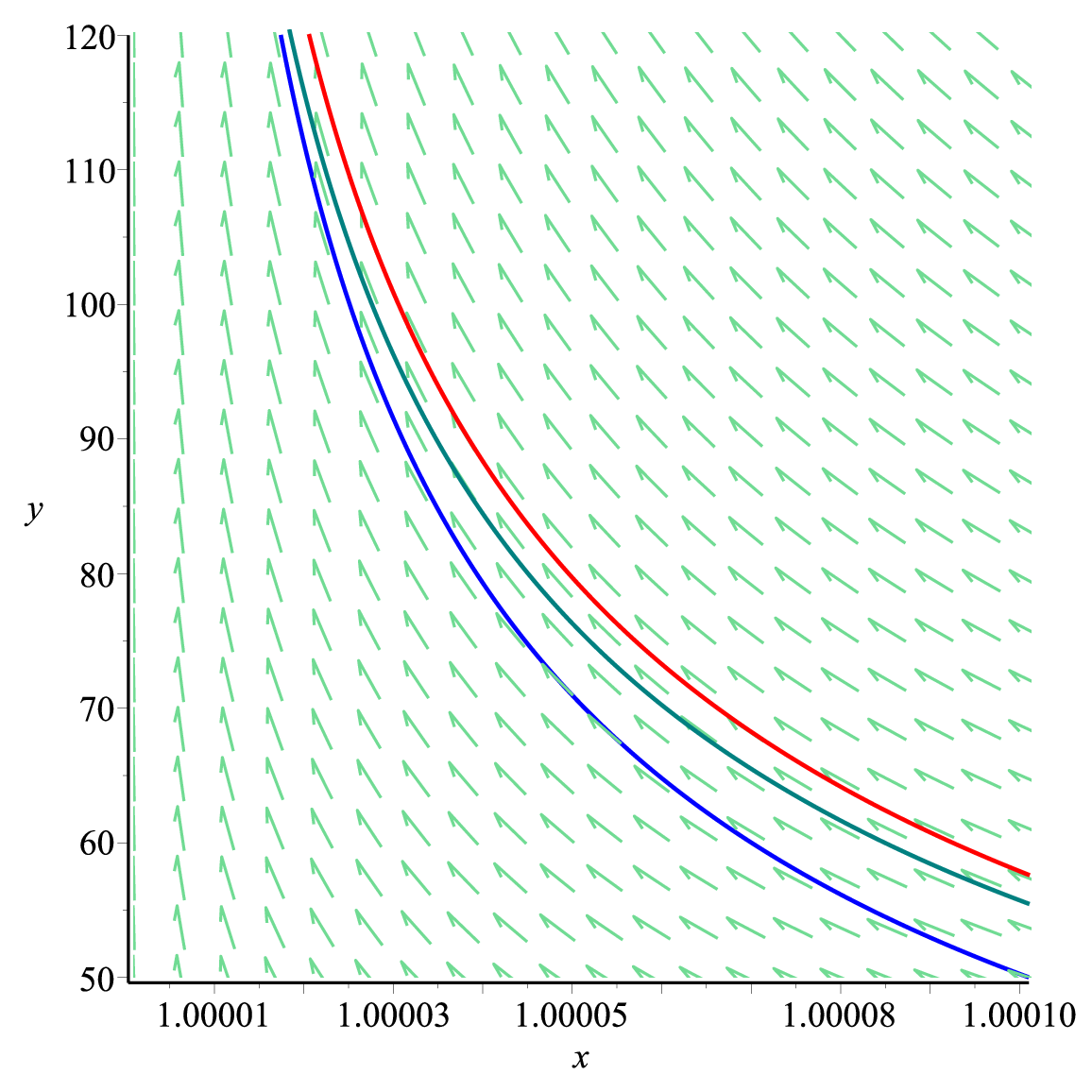}
\caption{The border curve $s_3'$ (in blue color), the initial point $O$, a Runge~-- Kutta approximation (in teal color) and an asymptotic approximation
to the solution of~\eqref{IV_problem} (in red color).}
\label{RK1}
\end{figure}

\begin{figure}[h]
\centering
\includegraphics[angle=0, width=0.45\textwidth]{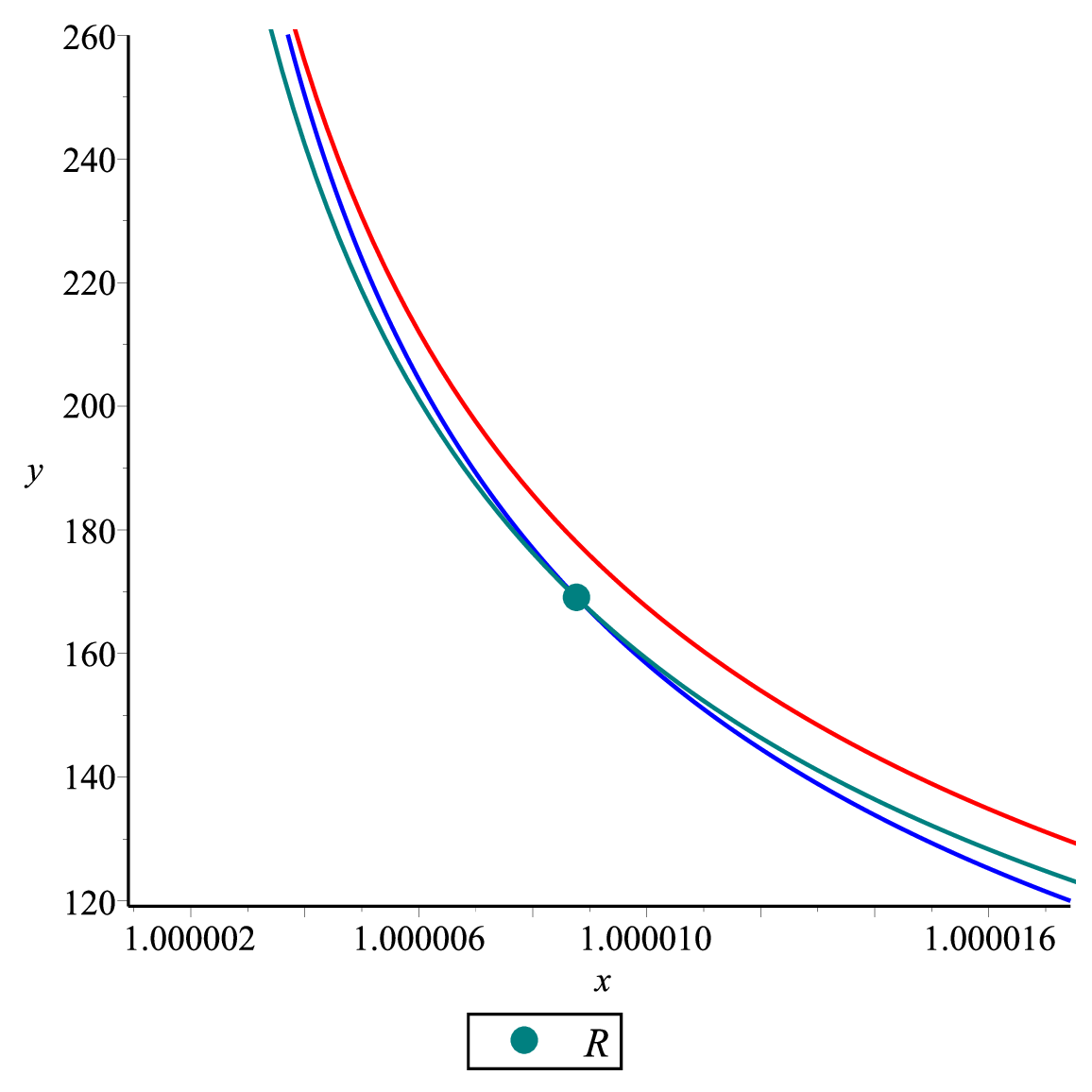}
\includegraphics[angle=0, width=0.45\textwidth]{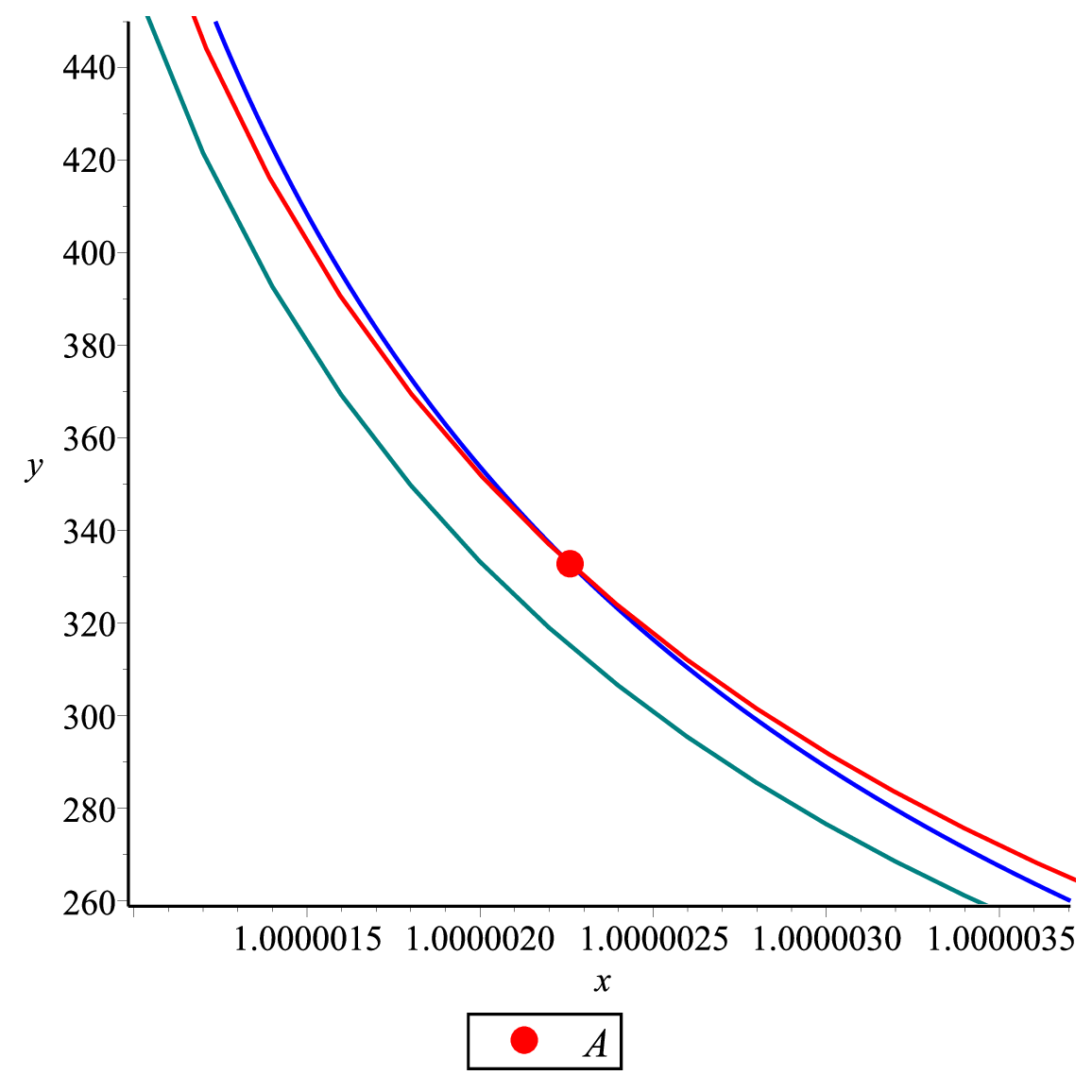}
\caption{The continuation of Figure~\eqref{RK1} to values $x\to 1+0$.}
\label{RK2}
\end{figure}

\bigskip

The author is indebted to Prof. Yu.\,G.~Nikonorov for helpful discussions of
the topic of this paper.

\vspace{10mm}

\bibliographystyle{amsunsrt}

\vspace{5mm}

\end{document}